\documentclass[12pt]{article}
\usepackage[psamsfonts]{amsfonts}
\usepackage{amsmath,amssymb,amsthm,dsfont,enumerate}
\usepackage{graphicx}               %voor het in de tekst opnemen van plaatjes

\usepackage[textsize=small]{todonotes}       % includes TO DO LIST
\usepackage{color}

\newtheorem{theorem}{Theorem}[section]

\newtheorem{definition}[theorem]{Definition}
\newtheorem{lemma}[theorem]{Lemma}

\newtheorem{proposition}[theorem]{Proposition}

\newtheorem*{remark}{Remark}

\newcommand{\eq}{\begin{equation}}
\newcommand{\en}{\end{equation}}
\newcommand{\nn}{\nonumber}

\newcommand{\ch}[1]{#1}
\newcommand{\col}[1]{#1}
\newcommand{\chs}[1]{#1}

\newcommand{\prob}{\mathbb P}
\newcommand{\expec}{\mathbb E}
\newcommand{\ind}{\mathds 1}

\newcommand{\atanh}{{\rm atanh}}

\newcommand{\dint}{{\rm d}}
\newcommand{\bbR}{\mathbb{R}}

\newcommand{\calT}{\mathcal{T}}

\newcommand{\sss}{\scriptscriptstyle}
\newcommand{\betahat}{\hat{\beta}}
\newcommand{\when}{{\rm when\ }}
\numberwithin{equation}{section}
\newcommand{\vep}{\varepsilon}
\newcommand{\e}{{\rm e}}
\newcommand{\rem}[1]{}

\def\1{{\mathchoice {1\mskip-4mu\mathrm l}      % Blackboard bold 1
{1\mskip-4mu\mathrm l}
{1\mskip-4.5mu\mathrm l} {1\mskip-5mu\mathrm l}}}
\newcommand{\indic}[1]{\1_{\{#1\}}}

\title  {Ising critical exponents on random trees and graphs}

\author
{
Sander Dommers
\footnote{
  Eindhoven University of Technology, Department of Mathematics and Computer Science,
  P.O.\ Box 513, 5600 MB Eindhoven, The Netherlands.
  E-mail: {\tt s.dommers@tue.nl, rhofstad@win.tue.nl}
}
\and Cristian Giardin\`a
\footnote{
  Modena and Reggio Emilia University, Department of Mathematics, Physics and Computer Science, via Campi 231/b,
41125 Modena, Italy. E-mail: {\tt cristian.giardina@unimore.it}
}
\and Remco van der Hofstad$\ ^*$
}

\date{\today}

\begin{document}

\maketitle

\begin{abstract}
We study the \ch{critical behavior of the} ferromagnetic Ising model on \ch{random trees as well as
so-called locally tree-like random graphs.}
\ch{We pay special attention to trees and graphs with a power-law offspring or degree distribution
whose tail behavior is characterized by its power-law exponent $\tau>2$.} \ch{We show that the critical
temperature of the Ising model equals the \col{inverse hyperbolic} tangent of the inverse of the mean offspring or mean forward
degree distribution. In particular, the \col{inverse} critical temperature equals zero when $\tau\in(2,3]$ where this mean
equals infinity.}

\ch{We further study the critical exponents $\boldsymbol{\delta}, \boldsymbol{\beta}$ and $\boldsymbol{\gamma}$,
describing how the (root) magnetization behaves close to criticality. We rigorously
identify these} critical exponents \ch{and show that they} take the values as predicted by Dorogovstev,
et al.~\cite{DorGolMen02} and Leone et al.~\cite{LeoVazVesZec02}. These values depend on
the power-law exponent $\tau$, taking the mean-field values for $\tau>5$, but different
values for $\tau\in(3,5)$.
\end{abstract}

%\todo{Remco: Discuss relation to random tree!}

\section{Introduction}
In the past decades complex networks and their behavior have attracted much attention. In the real world many of such networks can be found, for instance as social, information, technological and biological networks. An interesting property that many of them share is that they are {\em scale free}~\cite{New03}. This means that their degree sequences obey a {\em power law}, i.e., the fraction of nodes that have $k$ neighbors is proportional to $k^{-\tau}$ for some $\tau>1$. We therefore use power-law random graphs as a simple model for real-world networks.

Not only the structure of these networks is interesting, also the behavior of processes living on these networks \chs{is} a fascinating subject. Processes one can think of are opinion formation, the spread of information and the spread of viruses. An extensive overview of complex networks and processes on them is given by Newman in~\cite{New03}. It is especially interesting if these processes undergo a so-called {\em phase transition}, i.e., a minor change in the circumstances suddenly results in completely different behavior. Examples of such phase transitions include the sudden East European revolution in 1989 \cite{Kur91} and the unusual swine flu outbreak in 2009 \cite{CohEns09}.

Physicists have studied the behavior near phase transitions, the {\em critical behavior}, on complex networks for many different models, see~\cite{DorGolMen08} for an overview. Many of these results have not been mathematically rigorously proved. One of the few models for which rigorous results have been obtained is the contact process \cite{ChaDur09}, where the predictions of physicists, in fact, turned out not to be correct. A mathematical treatment of other models is therefore necessary.

We focus on the {\em Ising model}, a paradigm model for the study of phase transitions \cite{Nis05,Nis09,Nis11}. In this model a spin value that can be either $+1$ or $-1$ is assigned to every vertex. These spins influence each other with {\em ferromagnetic} interactions, i.e., neighboring spins prefer to be aligned. The strength of these interactions depends on the temperature. The first rigorous study of the Ising model on a random graph was performed by De Sanctis and Guerra
in \cite{SanGue08}, where the high and zero temperature regime of the Ising model on the Erd{\H o}s-R\'enyi random graph were analyzed. Later, in~\cite{DemMon10}, Dembo and Montanari analyzed the Ising model on locally tree-like random graphs with a finite-variance degree distribution for any temperature. In~\cite{DomGiaHof10}, we generalized these results to the case where the degree distribution has strongly finite mean, but possibly infinite variance, i.e., the degree distribution obeys a power-law with exponent $\tau>2$. An analysis of the critical behavior, however, was still lacking.

In this article, we rigorously study the critical behavior of the Ising model on power-law random graphs by computing various critical exponents. Predictions for the values of these exponents were given Dorogovtsev, et al.\ in~\cite{DorGolMen02} and independently by Leone et al.\ in~\cite{LeoVazVesZec02} and we prove that these values are indeed correct. These exponents depend on the power-law exponent $\tau$. We prove that \ch{the
critical exponents $\boldsymbol{\delta}, \boldsymbol{\beta}$ and $\boldsymbol{\gamma}$} take the classical mean-field values for $\tau>5$, and hence also for the Erd{\H o}s-R\'enyi random graph, but are different for $\tau\in(3,5)$. In~\cite{DorGolMen02,LeoVazVesZec02} also the case $\tau\in(2,3)$ is studied for which the critical temperature is infinite. Hence, the critical behavior should be interpreted as the temperature going to infinity, which is a different problem from approaching a finite critical temperature and is therefore beyond the scope of this article.

\ch{Our proofs always start by relating the magnetization \chs{of the Ising model on the random graph} and various of its derivatives to the root magnetization of a rooted random tree,
the so-called unimodular tree. After this, we identify the critical
exponents related to the root magetization on the rooted random tree. As a result, all our results also
apply to this setting, where only in the case of the regular tree, the mean-field critical exponents
have been identified \cite{Bax82}, and which we extend to general offspring distributions.}

\section{Model definitions and results}
\subsection{Ising model \ch{on finite graphs}}
We start by defining Ising models on finite graphs. Consider a random graph sequence $\ch{(G_n)}_{n \geq 1}$. Here $G_n=(V_n,E_n)$, with vertex set $V_n=[n] \equiv \{1,\ldots,n\}$ and with a random edge set $E_n$. To each vertex $i\in [n]$ an Ising spin $\sigma_i = \pm 1$ is assigned. A configuration of spins is denoted by $\sigma=(\sigma_i )_{i\in [n]}$. The {\em Ising model on $G_n$} is then defined by the Boltzmann\chs{-}\ch{Gibbs measure}
\eq\label{eq-boltzmann}
\mu_n(\sigma) = \frac{1}{Z_n(\beta, \underline{B})} \exp \left\{\beta \sum_{(i,j) \in E_n} \sigma_i \sigma_j + \sum_{i \in [n]} B_i \sigma_i\right\}.
\en
Here, $\beta \geq 0$ is the inverse temperature and $\underline{B}$ the vector of external magnetic fields $\underline{B}=(B_i)_{i \in [n]} \in \bbR^n$. For a uniform external field we write $B$ instead of $\underline{B}$, i.e., $B_i=B$ for all $i\in[n]$. The partition function $Z_n(\beta,\underline{B})$ is the normalization constant in~\eqref{eq-boltzmann}, i.e.,
\eq
Z_n(\beta,\underline{B}) = \sum_{\sigma \in \{-1,+1\}^n} \exp \left\{\beta \sum_{(i,j) \in E_n} \sigma_i \sigma_j + \sum_{i \in [n]} B_i \sigma_i\right\}.
\en
Note that the inverse temperature $\beta$ does not multiply the external field. This turns out to be technically convenient and does not change the results, because we are only looking at systems at equilibrium, and hence this would just be a reparametrization.

We let $\big<\cdot\big>_{\mu_n}$ denote the expectation with respect to the Ising measure $\mu_n$, i.e., for every bounded function $f: \{-1,+1\}^n \rightarrow \bbR$, we write
\eq
\big<f(\sigma)\big>_{\mu_n} =  \sum_{\sigma \in \{-1,+1\}^n} f(\sigma) \mu_n(\sigma).
\en

\subsection{Thermodynamics}
We study the critical behavior of this Ising model by analyzing the following \col{two} thermodynamic quantities:
\begin{definition}[Thermodynamic quantities] \rm
For a graph sequence $(G_n)_{n\geq1}$,
\begin{enumerate}[(a)]
\item let $M_n(\beta,B)=\frac{1}{n} \sum_{i\in[n]} \langle\sigma_i\rangle_{\mu_n}$ be the magnetization per vertex. Then, the thermodynamic limit of the {\em magnetization} per vertex equals
\eq
M(\beta, B) \equiv \lim_{n\rightarrow \infty} M_n(\beta,B).
\en
\item let $\chi_n(\beta,B)= \frac{1}{n} \sum_{i,j\in[n]} \left( \langle\sigma_i\sigma_j\rangle_{\mu_n}-\langle\sigma_i\rangle_{\mu_n}\langle\sigma_j\rangle_{\mu_n}\right)$ denote the susceptibility. Then, the thermodynamic limit of the {\em susceptibility} equals
\eq
\chi(\beta,B) \equiv \lim_{n\rightarrow\infty}\chi_n(\beta,B).
\en
%\todo{removed internal energy}
%\item let $U_n(\beta,B) = -\frac{1}{n} \sum_{(i,j)\in E_n} \langle\sigma_i \sigma_j\rangle_{\mu_n}$ be the internal energy per vertex. Then, the thermodynamic limit of the {\em specific heat} equals
%\eq
%\delta C(\beta,B) \equiv \lim_{n\rightarrow \infty} -\beta^2 \frac{\partial}{\partial \beta}U_n(\beta,B).
%\en
\end{enumerate}
\end{definition}
%\ch{It is not immediate that the above limits for $n\rightarrow \infty$ exist. For $M_n(\beta,B), \chi_n(\beta,B)$ and
%$U_n(\beta,B)$, this follows from \cite[Theorem 1.5]{DomGiaHof10},
%using the existence of the pressure per particle proved in \cite{DemMon10} and \cite[Theorem 1.4]{DomGiaHof10}.
%For the specific heat $\delta C(\beta,B)$ this is expected, but unproved.}
\col{The existence of  the above limits for $n\rightarrow \infty$ has been proved in  \cite[Theorem 1.5]{DomGiaHof10},
using the existence of the pressure per particle proved in \cite{DemMon10} and \cite[Theorem 1.4]{DomGiaHof10} \chs{ and using monotonicity properties}.}
We now define the critical temperature. We write $f(0^+)$ for $\lim_{x\searrow0}f(x)$.

\begin{definition}[Critical temperature] \rm
The critical temperature equals
\eq
\label{betac}
\beta_c \equiv \inf\{\beta: M(\beta,0^+)>0\}.
\en
\end{definition}
Note that such a $\beta_c$ can only exist in the thermodynamic limit, but not for the magnetization of a finite graph, since always $M_n(\beta,0^+)=0$. The critical behavior can now be expressed in terms of the following critical exponents. We write  $f(x) \asymp g(x)$ if the ratio $f(x)/g(x)$ is bounded away from 0 and infinity for the specified limit.

\begin{definition}[Critical exponents]\label{def-CritExp} \rm
The critical exponents $\boldsymbol{\beta,\delta,\gamma,\gamma'}$
%and $\boldsymbol{\alpha'}$
are defined by:
\begin{align}
M(\beta,0^+) &\asymp (\beta-\beta_c)^{\boldsymbol{\beta}},  &&{\rm for\ } \beta \searrow \beta_c; \label{eq-def-critexp-beta}\\
M(\beta_c, B) &\asymp B^{1/\boldsymbol{\delta}},  &&{\rm for\ } B \searrow 0;
\label{eq-def-critexp-delta}\\
\chi(\beta, 0^+) &\asymp (\ch{\beta_c-\beta})^{-\boldsymbol{\gamma}},  &&{\rm for\ } \beta \nearrow \beta_c; \label{eq-def-critexp-gamma}\\
\chi(\beta, 0^+) &\asymp (\beta-\beta_c)^{-\boldsymbol{\gamma'}},  &&{\rm for\ } \beta \searrow \beta_c.\label{eq-def-critexp-gamma'}
%; \\
%\delta C(\beta, 0^+) &\asymp (\beta-\beta_c)^{-\boldsymbol{\alpha'}},  &&{\rm for\ } \beta \searrow \beta_c.
\end{align}
%\todo{Sander: Correct definition of critical exponent $\boldsymbol{\alpha'}$?}

\begin{remark}
We emphasize that there is a difference between the symbol $\beta$ for the inverse temperature and the bold symbol $\boldsymbol{\beta}$ for the critical exponent in~\eqref{eq-def-critexp-beta}. Both uses for $\beta$ are standard in the literature, so we decided to stick to this notation.

Also note that these are stronger definitions than usual. E.g., normally the critical exponent $\boldsymbol{\beta}$ is defined as that value such that
\eq
M(\beta,0^+) = (\beta-\beta_c)^{\boldsymbol{\beta}+o(1)},
\en
where $o(1)$ is a function tending to zero for $\beta\searrow\beta_c$.
\end{remark}
\end{definition}

\subsection{\ch{Locally tree-like random graphs}}
We study the critical behavior of the Ising model on graph sequences $(G_n)_{n\geq1}$ that are assumed to be {\em locally like a homogeneous random tree}, to have a {\em power-law degree distribution} and to be {\em uniformly sparse}. We give the formal definitions of these assumptions below, but we first introduce some notation.

Let the random variable $D$ have distribution $P=(p_k)_{k\geq1}$, i.e., $\prob[D=k]=p_k,$ for $k=1,2,\ldots$. We define its {\em forward degree distribution} by
\eq \label{eq-defrho}
\rho_k = \frac{(k+1) p_{k+1}}{\expec[D]},
\en
where we assume that $\expec[D]<\infty$. Let $K$ be a random variable with $\prob[K=k]=\rho_k$ and write $\nu=\expec[K]$. The random rooted tree $\calT(D,K,\ell)$ is a branching process with $\ell$ generations, where the root offspring is distributed as $D$ and the vertices in each next generation have offsprings that are independent of the root offspring and are {\em independent and identically distributed} (i.i.d.) copies of the random variable $K$. We write $\calT(K,\ell)$ when the offspring at the root has the same distribution as $K$.
%\todo{Do we ever use $\calT(K,\ell)$?}

We write that an event $\mathcal{A}$ holds \emph{almost surely} (a.s.) if $\prob[\mathcal{A}]=1$.
The ball of radius $r$ around vertex $i$, $B_i(r)$, is defined as the graph induced by the vertices at graph distance at most $r$ from vertex $i$.
For two rooted trees $\calT_1$ and $\calT_2$, we write that $\calT_1 \simeq \calT_2$, when there exists a bijective map from the vertices of $\calT_1$ to those of $\calT_2$ that preserves the adjacency relations.

\begin{definition}[Local convergence to homogeneous \col{random} trees] \label{ass-convtree} %\rm
Let $\prob_n$ denote the law induced on the ball $B_i(t)$ in $G_n$ centered at a uniformly chosen vertex $i\in[n]$. We say that the graph sequence $(G_n)_{n\geq 1}$ is {\em locally tree-like} with asymptotic degree \ch{distribution $P$} when, for any rooted tree $\calT$ with $t$ generations
\eq
\lim_{n\rightarrow\infty} \prob_n [B_i(t) \simeq \calT] = \prob[\calT(D,K,t) \simeq \calT].
\en
\end{definition}
Note that this implies \ch{in particular} that the degree of a uniformly chosen vertex of the graph has an asymptotic
degree distributed as $D$.

\begin{definition}[Uniform sparsity] \label{ass-unisparse} %\rm
We say that the graph sequence $(G_n)_{n \geq 1}$ is {\em uniformly sparse} when, a.s.,
\eq
\lim_{\ell\rightarrow\infty} \limsup_{n\rightarrow\infty} \frac{1}{n} \sum_{i \in [n]} D_i \ind_{\{ D_i \geq \ell\}} = 0,
\en
where $D_i$ is the degree of vertex $i$ and $\ind_{\mathcal{A}}$ denotes the indicator of the event $\mathcal{A}$.
\end{definition}

Note that uniform sparsity follows if $\frac{1}{n} \sum_{i \in [n]} D_i\to \expec[D]$ a.s.,
by the weak convergence of the degree of a uniform vertex.

\ch{We pay special attention to cases where the degree distribution satisfies a power law, as defined
in the following definition. For power-law degree distributions, not all moments of the degrees are finite,
which has severe consequences for the critical behavior of the Ising model.}

\begin{definition}[Power laws] \label{ass-degdist} %\rm
We say that the distribution $P=(p_k)_{k\geq1}$ obeys a {\em power law with exponent $\tau$} when there exist constants $C_p>c_p>0$ such that, for all $k=1,2,\ldots$,
	\eq \label{eq-ppowerlaw}
	c_p k^{-(\tau-1)} \leq \sum_{\ell\geq k} p_\ell \leq C_p k^{-(\tau-1)}.
	\en
\end{definition}

\ch{\subsection{The random Bethe tree}
We next extend our \col{definitions} to the random tree $\calT(D,K,\infty)$, which is an infinite random tree.
One has to be very careful in defining a Gibbs measure on this tree,
since trees suffer from the fact that the boundaries of intrinsic \chs{(i.e., graph distance)} balls in them have \chs{a} size that is comparable to their volume.
We can adapt the construction of the Ising model on the regular tree in \cite{Bax82} to this setting, as we now explain.
For $\beta\geq 0, B\in {\mathbb R}$, let $\mu_{t,\beta,B}^{+/f}$ be the Ising model on $\calT(D,K,t)$ with
$+$ respectively free boundary conditions.
For a function $f$ that only depends on $\calT(D,K,m)$ with $m\leq t$, we let
	\eq
	\langle f\rangle_{\mu_{\beta,B}^{+/f}}=\lim_{t\rightarrow \infty} \langle f\rangle_{\mu_{t,\beta,B}^{+/f}}.
	\en
Below, we argue that these limits indeed exist and are equal for $B>0$.
This defines a unique infinite volume Gibbs measure $\mu_{\beta,B}^{+/f}$ on the random Bethe tree.
The quantity $M(\beta,B)$ is the expected root magnetization for this infinite volume Gibbs measure on the random Bethe
tree. Our results also apply to this setting under the assumption that the degree of the root obeys a power law
in \eqref{eq-ppowerlaw} or that $\expec[K^3]<\infty$.
The critical value $\beta_c$ for the root magnetization is again defined by \eqref{betac}.}

\subsection{Main results}
We now present our main results which describe the critical behavior of the Ising model on
power-law random graphs \col{and random trees with power-law offspring distribution}.
We first give an expression for the critical temperature:
\begin{theorem}[Critical temperature]
\label{thm-CritTemp}
Assume that the random graph sequence $(G_n)_{n\geq1}$ is locally tree-like with asymptotic degree
distribution $P$ \chs{and} is uniformly sparse.
Then, a.s., the critical temperature $\beta_c$ \ch{of $(G_n)_{n\geq1}$ and of the random Bethe tree $\calT(D,K,\infty)$ equal\chs{s}}
	\eq
	\beta_c=\atanh(1/\nu).
	\en
\end{theorem}
%\todo{$\tau\in(2,3)$ has been moved to the discussion section}
%\ch{If} $\tau\in(2,3)$, then $\nu=\infty$ and hence $\beta_c=0$. Since \ch{in this case} there is no phase transition at finite temperature, we only %investigate the critical behavior for cases where $\nu<\infty$.
Near the critical temperature the behavior of the Ising model can be described by critical exponents. The values of these critical exponents for different values of $\tau$ are stated in the following theorem:

\begin{theorem}[Critical exponents] \label{thm-CritExp}
Assume that the random graph sequence $(G_n)_{n\geq1}$ is locally tree-like with asymptotic degree distribution $P$ that obeys $\ch{\expec[K^3]<\infty}$ or a power law with exponent $\tau\in(3,5]$, and is uniformly sparse, \ch{or that the random Bethe tree obeys $\expec[K^3]<\infty$ or a power law with exponent $\tau\in(3,5]$.} Then, the critical exponents $\boldsymbol{\beta,\delta}$ and $\boldsymbol{\gamma}$ defined in Definition~\ref{def-CritExp} \ch{exist and
satisfy}
\begin{center}
{\renewcommand{\arraystretch}{1.2}
\renewcommand{\tabcolsep}{1cm}
\begin{tabular}[c]{c|ccc}
 &  $\tau\in(3,5)$ & $\expec[K^3]<\infty$    \\
\hline
$\boldsymbol{\beta}$  & $1/(\tau-3)$ & $1/2$ \\
$\boldsymbol{\delta}$ & $\tau-2$ & $3$\\
$\boldsymbol{\gamma}$%, \boldsymbol{\gamma'}$
& $1$ & $1$ \\
\ch{$\boldsymbol{\gamma'}$}
& \ch{$\geq 1$} & \ch{$\geq 1$}\\
%$\boldsymbol{\alpha'}$ & $(\tau-5)/(\tau-3)$ & $0$ \\
\end{tabular}
}
\end{center}
For the boundary case $\tau=5$ there are logarithmic corrections for $\boldsymbol{\beta}=1/2$ and $\boldsymbol{\delta}=3$, but not for $\boldsymbol{\gamma}=1$ \ch{and for the lower bound $\boldsymbol{\gamma'}\geq 1$.} Indeed, \eqref{eq-def-critexp-gamma}
holds with $\boldsymbol{\gamma}=1$ \ch{and the lower bound in \eqref{eq-def-critexp-gamma'}
holds with $\boldsymbol{\gamma'}=1$,} while
	\eq
	\label{log-corr-M-tau5}
	M(\beta,0^+) \asymp \Big(\frac{\beta-\beta_c}{\log{1/(\beta-\beta_c)}}\Big)^{1/2} \quad {\rm for\ } \beta \searrow \beta_c,
	\qquad
	M(\beta_c, B) \asymp \Big(\frac{B}{\log(1/B)}\Big)^{1/3} \quad {\rm for\ } B \searrow 0.
	\en
\end{theorem}

\ch{Unfortunately, we cannot prove that the critical exponent $\boldsymbol{\gamma'}$ exists, see the discussion
in the next section for more details on this issue.}

\subsection{Discussion and open problems}
\label{sec-disc-op}
In this section, we discuss relations to the literature, possible extensions and open problems.

\paragraph{The Ising model on random trees and random graphs.}
\col{A key idea to analyze \chs{the} Ising model on random graphs is to use the fact that expectations of
local quantities coincide
with the correspond\chs{ing} values for the Ising model on suitable random trees  \cite{DemMon10}.
Statistical mechanics models on deterministic trees have been studied extensively in the
literature (see for instance \cite{Bax82,Lyo89} \chs{and its relation to}
``broadcasting on trees'' \chs{in }\cite{Evans00,MezMon06}\chs{)}.
The analysis on random trees is more recent and has been triggered by the study
of models on random graphs. Extensions beyond the Ising model,
e.g.\chs{, the} Potts model, pose new challenges \cite{Dem12}.
}
%\todo{Refer to \cite{DemMon10}}
%\todo{Refer to \cite{Lyo89} here as well!}
%\todo{Refer to regular trees!}
%\todo{Refer to `broadcasting on trees'}

\paragraph{Relation to the physics literature.}
Theorem \ref{thm-CritExp} confirms the predictions in~\cite{DorGolMen02,LeoVazVesZec02}.
\col{For $\tau\leq 3$, one has $\nu=\infty$ and hence $\beta_c=0$ by Theorem \ref{thm-CritTemp}, so that the critical behavior
coincides with the infinite temperature limit. Since \ch{in this case} there is no phase transition at finite temperature,
we do not study the critical behavior \chs{here}.
For $\tau=5$,} in \cite{DorGolMen02},
also the logarithmic correction for $\boldsymbol{\beta}=1/2$ \ch{in \eqref{log-corr-M-tau5}
is computed,} but not that of $\boldsymbol{\delta}=3$.

\paragraph{The critical exponents $\boldsymbol{\gamma'}$ and other critical exponents.}
\col{
\ch{Theorem \ref{thm-CritExp} only gives a lower bound on the critical exponent $\boldsymbol{\gamma'}$.
%and does not study other  critical exponents.
It is predicted that $\boldsymbol{\gamma'}=1$ for all $\tau>3$, while
there are also predictions for  other critical exponents.
For instance the critical exponent $\boldsymbol{\alpha'}$ for the specific heat in the low-temperature phase
satisfies $\boldsymbol{\alpha'}=0$ when $\expec[K^3]<\infty$ and
$\boldsymbol{\alpha'}=(\tau-5)/(\tau-3)$ in the power-law case with $\tau\in(3,5)$
(see \cite{DorGolMen02,LeoVazVesZec02}).
We prove the lower bound $\boldsymbol{\gamma'}\geq 1$ in Section \ref{sec-gamma'}
below, and we also present a heuristic argument that $\boldsymbol{\gamma'}\leq 1$ holds.
The critical exponent $\boldsymbol{\alpha'}$ for the specific heat is beyond our current methods,
partly since we are not able to relate the specific heat on a random graph to that on the random Bethe tree.}
%\todo{Mention other critical exponents? Correlation length?}
}

\paragraph{Light tails.}
\ch{The case $\expec[K^3]<\infty$ includes all power-law degree distributions with $\tau>5$,
but also cases where $P$ does {\em not} obey a power law. This means, e.g., that Theorem \ref{thm-CritExp}
also identifies the  critical exponents for the Erd{\H o}s-R\'enyi random graph where the degrees have an asymptotic
Poisson distribution.}

\paragraph{Inclusion of slowly varying functions.}
In Definition \ref{ass-degdist}, we have assumed that the asymptotic degree distribution
obeys a perfect power law. Alternatively, one could assume that
$\sum_{\ell\geq k} p_\ell \asymp L(k)k^{-(\tau-1)}$ for some function $k\mapsto L(k)$
that is slowly varying at $k=\infty$.
For $\tau>5$ and any slowly varying function,
we still have $\expec[K^3]<\infty$, so the results do not change and Theorem \ref{thm-CritExp}
remains to hold. For $\tau\in(3,5]$, we expect
slowly varying corrections to the critical behavior in Theorem \ref{thm-CritExp}.
For example, $\expec[K^3]<\infty$ for $\tau=5$ and $L(k)=(\log{k})^{-2}$,
so that the logarithmic corrections present for $\tau=5$ disappear.

\ch{\paragraph{Beyond the root magnetization for the random Bethe tree.}
We have identified the critical value and some critical exponents for the root magnetization
on the random Bethe tree. The random Bethe tree is a so-called \emph{unimodular} graph,
which is a rooted graph that often arises as the local weak limit of a sequence of graphs
(in this case, the random graphs $(G_n)_{n\geq 1}$). See \cite{AldLyo07, BenLyoSch12}
for more background on unimodular graphs and trees, in particular, $\calT(D,K,\infty)$ is the
so-called \emph{unimodular
Galton-Watson tree} as proved by Lyons, Pemantle and Peres in
\cite{LyoPemPer95}. One would expect that the
magnetization of the graph, which can be defined by
	\eq
	M_T(\beta,B)=\lim_{t\rightarrow \infty} \frac{1}{|B_{\phi}(t)|}\sum_{v\in B_{\phi}(t)} \sigma_v,
	\en
where $B_{\phi}(t)$ is the graph induced by vertices at graph distance at most
$t$ from the root $\phi$ and
$|B_{\phi}(t)|$ is the number of elements in it, also converges a.s.\ to a limit.
However, we expect that $M_T(\beta,B)\neq M(\beta,B)$ due to the special role of
the root $\phi$, which vanishes in the above limit. Thus one would expect to believe that
$M_T(\beta,B)$ equals the root magnetization of the tree where each vertex has degree
distribution $K+1$. Our results show that also $M_T(\beta,B)$ has the same critical temperature
and critical exponents as $M(\beta,B)$.
}

\ch{\paragraph{Relation to the Curie-Weiss model.}}
%\todo{Add this discussion!}
Our results show that locally tree-like random graphs with finite fourth moment of the degree distribution are
in the same universality class as the mean-field model on the complete graph, which is the Curie-Weiss model.
We further believe that the Curie-Weiss model should enter as the limit of $r\rightarrow \infty$ for the $r$-regular
random graph, in the sense that these have the same critical exponents (as we already know), as well as that all
constant\chs{s} arising in asymptotics match up nicely \col{(cf. the discussion at the end of Section \ref{sec-gamma'})}.
Further, our results show that for $\tau\in(3,5]$, the Ising model
has \emph{different} critical exponents \chs{than} the ones for the Curie-Weiss model, so these constitute a set of different
universality classes.

\paragraph{Organization of the article.}
The remainder of this article is organized as follows. We start with some preliminary computations
in Section~\ref{sec-prel}. In Section~\ref{sec-CritTemp} we prove that the critical temperature is as stated in Theorem~\ref{thm-CritTemp}.
The proof that the exponents stated in Theorem~\ref{thm-CritExp} are indeed the correct values of
$\boldsymbol{\beta}$ and $\boldsymbol{\delta}$ is given in Section~\ref{sec-CritExpBetaDelta}.
\ch{The value of $\boldsymbol{\gamma}$ %and $\boldsymbol{\gamma'}$
is identified in Section~\ref{sec-CritExpChi}, where also the lower bound on
$\boldsymbol{\gamma'}$ is proved and a heuristic is presented for the matching upper bound.}
%, and the value of $\boldsymbol{\alpha'}$ is identified in Section~\ref{sec-CritExpSpecHeat}.

\section{Preliminaries}
\label{sec-prel}
An important role in our analysis is played by the distributional recursion
\eq \label{eq-recursion}
h^{(t+1)} \stackrel{d}{=} B + \sum_{i=1}^{K_t} \xi(h_i^{(t)}),
\en
where
\eq \label{eq-defxi}
\xi(h) = \atanh (\betahat\tanh(h)),
\en
with $\betahat=\tanh(\beta)$, and where $h^{(0)} \equiv B$, $(K_t)_{t \geq 1}$, are i.i.d.\ with distribution $\rho$
and $(h_i^{(t)})_{i\geq1}$ are i.i.d.\ copies of $h^{(t)}$ independent of $K_t$. In \cite[Proposition~1.7]{DomGiaHof10},
we have proved that this recursion has a unique fixed point $h$ for all $\beta\geq0$ and $B>0$.
%\ch{This fixed point $h$ corresponds to the expected root magnetization on the random Bethe tree.}
% MOVED TWO LINES BELOW
Whenever we write $h$ or $h_i$ this is a random variable distributed as the fixed point of~\eqref{eq-recursion}. Since $h$ is a fixed point, we can interchange $h \stackrel{{\rm d}}{=}B+\sum_{i=1}^K\xi(h_i)$ in expectations and we often do this.

\col{This fixed point $h$ yields  the random field acting on the root of the random Bethe tree \col{$\calT(D,K,\infty)$} due to
its offsprings. In particular we can use the fixed point $h$ to give an explicit expression for the magnetization:}

\begin{proposition}[Magnetization]\label{prop-Magnetization}
Assume that the random graph sequence $(G_n)_{n\geq1}$ is locally tree-like with asymptotic degree distribution $P$ that obeys $\ch{\expec[K]<\infty}$
or a power law with exponent $\tau\in(2,3)$ and is uniformly sparse. Then, a.s., for all $\beta\geq0$ and $B>0$, the thermodynamic limit of the {\em magnetization} per vertex exists and is given by
\eq
M(\beta, B) = \expec\Big[ \tanh\Big(B+\sum_{i=1}^{D} \xi(h_i)\Big)\Big],
\en
where
\begin{itemize}
\item[\rm (i)]
$D$ has distribution $P$;
\item[\rm (ii)]
$(h_i)_{i\geq1}$ are i.i.d.\ copies of the fixed point of the distributional recursion~\eqref{eq-recursion};
\item[\rm (iii)]
$D$ and $(h_i)_{i\geq1}$ are independent.
\end{itemize}
\ch{The same holds on the random Bethe tree $\calT(D,K,\infty)$.}
\end{proposition}

This \ch{proposition} was proved in~\cite[\ch{Corollary 1.6(a)}]{DomGiaHof10} by differentiating the expression for the
thermodynamic limit of the pressure per particle that was first obtained. Here we present a more intuitive proof:
\begin{proof}[Proof of Proposition~\ref{prop-Magnetization}]
Let $\ch{\phi}$ be a vertex picked uniformly at random from $[n]$ and $\expec_n$ be the corresponding expectation. Then,
\eq
M_n(\beta,B) = \frac1n \sum_{i=1}^n \langle\sigma_i\rangle = \expec_n[\langle\sigma_{\ch{\phi}}\rangle].
\en
Denote by $\langle\cdot\rangle^{\ell,+/f}$ the \col{expectations with respect to the} Ising measure with $+$/free boundary conditions on vertices at graph distance $\ell$ from $\ch{\phi}$. Note that $\langle \sigma_{\ch{\phi}}\rangle^{\ell,+/f}$ only depends on the spins of vertices in $B_{\ch{\phi}}(\ell)$. By the GKS inequality~\cite{KelShe68},
\eq
\langle\sigma_{\ch{\phi}}\rangle^{\ell,f}\leq\langle\sigma_{\ch{\phi}}\rangle\leq\langle\sigma_{\ch{\phi}}\rangle^{\ell,+}.
\en
Taking the limit $n\rightarrow\infty$, the ball $B_{\ch{\phi}}(\ell)$ has the same distribution as the random tree $\calT(D,K,\ell)$, because of the locally tree-like nature of the graph sequence. Conditioned on the tree $\calT$, we can prune the tree, see~\cite[Lemma~4.1]{DemMon10}, to obtain that
\eq
\langle\sigma_{\ch{\phi}}\rangle^{\ell,f} = \tanh\Big(B+\sum_{i=1}^{D} \xi(h_i^{(\ell-1)})\Big).
\en
Similarly,
\eq
\langle\sigma_{\ch{\phi}}\rangle^{\ell,+} = \tanh\Big(B+\sum_{i=1}^{D} \xi(h_i^{'(\ell-1)})\Big),
\en
where $h_i^{'(t+1)}$ also satisfies~\eqref{eq-recursion}, but has initial value $h^{'(0)}=\infty$. Since this recursion has a unique fixed point~\cite[Prop\chs{osition}~1.7]{DomGiaHof10}, we prove the proposition by taking the limit $\ell\rightarrow\infty$ and taking the expectation over the tree $\calT(D,K,\infty)$.
\end{proof}

To study the critical behavior we \chs{investigate the function $\xi(x)=\atanh(\betahat \tanh x)$ and} prove two important bounds that play a crucial role throughout this paper:
\begin{lemma}[\ch{Properties of $x\mapsto \xi(x)$}]
\label{lem-boundatanh}
For all $x,\beta\geq0$,
\eq
\betahat x -\frac{\betahat}{3(1-\betahat^2)}x^3 \leq \xi(x) \leq \betahat x.
\en
%where
%\eq
%\xi(x) \equiv \atanh(\betahat \tanh x).
%\en
The upper bound holds with strict inequality if $x,\beta>0$.
\end{lemma}
\begin{proof}
By Taylor's theorem,
\eq
\xi(x) = \xi(0)+\xi'(0)x+\xi''(\zeta)\frac{x^2}{2},
\en
for some $\zeta \in (0,x)$. It is easily verified that $\xi(0)=0$,
\eq
\xi'(0) = \frac{\betahat (1-\tanh^2 x)}{1-\betahat^2\tanh^2 x}\bigg|_{x=0} = \betahat,
\en
and
\eq\label{eq-secondderxi}
\xi''(\zeta) = -\frac{2 \betahat (1-\betahat^2) (\tanh \zeta) (1-\tanh^2 \zeta)}{(1-\betahat^2\tanh^2 \zeta)^2} \leq 0,
\en
thus proving the upper bound. If $x,\beta>0$ then also $\zeta>0$ and hence the above holds with strict inequality.

For the lower bound, note that $\xi''(0)=0$ and
\begin{align}\label{eq-thirdderxi}
\xi'''(\zeta) &= -\frac{2 \betahat (1-\betahat^2) (1-\tanh^2 \zeta)}{(1-\betahat^2\tanh^2 \zeta)^3}\left(1-3(1-\betahat^2)\tanh^2\zeta-\betahat^2\tanh^4\zeta\right) \nonumber\\
&\geq -\frac{2 \betahat(1-\betahat^2)(1-\tanh^2\zeta)}{(1-\betahat^2)^2(1-\tanh^2\zeta)}=-\frac{2 \betahat}{1-\betahat^2}.
\end{align}
Thus, for some $\zeta\in(0,x)$,
\eq
\xi(x) = \xi(0)+\xi'(0)x +\xi''(0)\frac{x^2}{2}+ \xi'''(\zeta)\frac{x^3}{3!} \geq \betahat x -\frac{2 \betahat}{1-\betahat^2}\frac{x^3}{3!}.
\en
\end{proof}

We next study tail probabilites of $(\rho_k)_{k\geq 0}$. Here, for a probability distribution
$(q_k)_{k\geq 0}$ on the integers, we write $q_{\geq k}=\sum_{\ell \geq k}q_\ell$.

\begin{lemma}[Tail probabilities of $(\rho_k)_{k\geq 0}$]
\label{lem-tail-rho}
Assume that \eqref{eq-ppowerlaw} holds for some $\tau>2$.
Then, \col{for the size-biased distribution defined in (\ref{eq-defrho})},
there exist $\col{0 <}  c_{\rho}\leq C_{\rho}$ such that, for all $k\geq 1$,
	\eq
	c_{\rho} k^{-(\tau-2)}\leq \rho_{\geq k} \leq C_{\rho} k^{-(\tau-2)}.
	\en
\end{lemma}

\begin{proof}
The lower bound follows directly from the fact that $\rho_{\geq k}\geq \ch{ (k+1)p_{\geq k+1}/\expec[D]}$,
and \eqref{eq-ppowerlaw}.
For the upper bound, we note that for any probability distribution $(q_k)_{k\geq 0}$ on the \col{non-negative} integers,
we have the partial summation identity
	\eq
	\label{partial-summation}
	\sum_{k\geq 0} q_k f(k) =f(0)+\sum_{\ell\geq 1} q_{\geq \ell} [f(\ell)-f(\ell-1)],
	\en
provided that either $f(k)q_{\geq k}\rightarrow 0$ \ch{as $k\rightarrow \infty$ or
$k\mapsto f(k)$} is either non-decreasing or non-increasing. Indeed,
	\eq
	\sum_{k\geq 0} q_k f(k)=f(0)+\sum_{k\geq 0} q_k [f(k)-f(0)]	
	=\ch{f(0)+}\sum_{k\geq 0} q_k \sum_{\ell=1}^k [f(\ell)-f(\ell-1)].
	\en
Interchanging the summation order (which is allowed by our assumptions) p\chs{r}ovides the proof.

We start by proving bounds on $\rho_{\geq k}$. We rewrite
	\eq
	\rho_{\geq k} = \sum_{\ell\geq k} \frac{(\ell+1)p_{\ell+1}}{\expec[D]}
	=\sum_{\chs{\ell}\geq 0} f(\ell)p_{\ell+1},
	\en
where $f(\chs{\ell})=(\ell+1)\indic{\chs{\ell}\geq k}/\expec[D]$. By \eqref{partial-summation}
\col{with $q_{\chs{\ell}}=p_{\chs{\ell}+1}$},
\ch{for $k\geq 1$ so that $f(0)=0$,}
	\eq
	\rho_{\geq k} =\sum_{\col{\ell\geq 1}} [f(\ell)-f(\ell-1)]p_{\geq \ell+1} =\frac{\col{(k+1)}p_{\geq k+1}}{\expec[D]}
	+\frac{1}{\expec[D]}\sum_{\ell\geq \col{k+1}} p_{\geq \ell+1}.
	\en
From~\eqref{eq-ppowerlaw}, it follows that
	\eq
	\label{rho-bound}
	\rho_{\geq k}
	%=\sum_{\ell\geq 0} [f(\ell)-f(\ell-1)]p_{\geq \ell+1}
	\leq \frac{\ch{C_p}}{\expec[D]}(k+1)^{-(\tau-2)}
	+\sum_{\ell\geq k+1} \frac{\ch{C_p}}{\expec[D]}(\ell+1)^{-(\tau-\col{1})},
	\en
so that there exists a constant $C_{\rho}$ such that
	\eq
	\label{rho-tail-bds}
	\rho_{\geq k} \leq C_{\rho} k^{-(\tau-2)}.
	\en
\end{proof}

When computing the critical exponents for $\tau\in(3,5]$, we often split the analysis into two cases: one where $K$ is small
and one where $K$ is large. For this we need bounds on truncated moments of $K$ \ch{which are the content of the next lemma.}

\begin{lemma}[Truncated moments of $K$]
\label{lem-truncmoment}
\col{Assume that \eqref{eq-ppowerlaw} holds for some $\tau>2$. Then}
there exist \ch{constants} $C_{a,\tau} \col{>0}$ such that,
\ch{as $\ell \rightarrow \infty,$}
	\eq
	\expec\left[K^a \ind_{\{K \leq \ell\}}\right] \leq
	\begin{cases}
	C_{a,\tau}\ell^{a-(\tau-2)} &\text{\ch{when} } a > \tau-2,\\
	C_{\tau-2,\tau}\log \ell &\text{\ch{when} }  a=\tau-2.
	\end{cases}
	\en
\ch{and, when $a<\tau-2$,}
	\eq
	\expec\left[K^a \ind_{\{K > \ell\}}\right] \leq C_{a,\tau}\ell^{a-(\tau-2)}.
	\en
Finally, when $\tau=5$, there exists a constant $c_{3,5} \col{>0}$ such that,
\ch{as $\ell \rightarrow \infty,$}
	\eq
	\label{log-correct-lb-tau5}
	\expec\left[K(K-1)(K-2)\indic{K\leq \ell}\right]\geq c_{3,5}\log{\ell}.
	\en
\end{lemma}

\begin{proof}
We start by bounding the truncated moments of $K$. We rewrite, using
\eqref{partial-summation} and with $f(k)=k^a \indic{k\leq \ell}$,
	\eq
	\expec\left[K^a \indic{K \leq \ell}\right]
%	=\sum_{k=0}^{\ell} k^a \rho_k
	=\sum_{k=0}^{\infty} f(k)\rho_k
	=\sum_{k=\col{1}}^{\infty} [f(k)-f(k-1)]\rho_{\geq k}
	\leq \sum_{k=\col{1}}^{\col{\lfloor{\ell}\rfloor}}[k^a-(k-1)^a]\rho_{\geq k}.
	\en
Using $k^a-(k-1)^{a}=a\int_{k-1}^k x^{a-1}dx\leq a k^{a-1}$, we arrive at
	\eq
	\expec\left[K^a \indic{K \leq \ell}\right] \leq a C_{\rho}
	\sum_{k=1}^{\lfloor\ell\rfloor} k^{a-1} \col{k}^{-(\tau-2)}
	\leq a C_{\rho}\sum_{k=\col{1}}^{\lfloor\ell\rfloor+1} k^{a-(\tau-1)}.
	\en
Note that \ch{$k\mapsto k^{a-(\tau-1)}$ is either increasing or decreasing.} Hence,
	\eq
	\sum_{k=\col{1}}^{\lfloor\ell\rfloor+1} k^{a-(\tau-1)} \leq \int_{1}^{\ell+2} k^{a-(\tau-1)} \dint k.
	\en
For $a>\tau-2$,
	\eq
	\int_{1}^{\ell+2} k^{a-(\tau-1)} \dint k \leq \ch{\frac{2}{a+2-\tau}\ell^{a-(\tau-2)},}
	\en
whereas for $a=\tau-2$,
	\eq
	\int_{1}^{\ell+2} k^{a-(\tau-1)} \dint k \leq 2 \log \ell.
	\en
Similarly, for $a<\tau-2$,
	\begin{align}
	\expec\left[K^a \indic{K > \ell}\right] &=\lceil \ell \rceil^{a} \rho_{\geq \ell}
	+\sum_{k>\ell} [k^a-(k-1)^a] \rho_{\geq k}\\
	&\leq C_{\rho}\lceil \ell \rceil^{a-(\tau-2)} +
	 aC_{\rho}\sum_{\lfloor\ell\rfloor+1}^\infty k^{\ch{a-1}} (k+1)^{-(\tau-2)} \leq C_{a,\tau} \ell^{a-(\tau-2)}.\nn
	\end{align}
Finally, we prove \eqref{log-correct-lb-tau5}, for which we compute with $f(k)=k(k-1)(k-2)$,
	\eq
	\expec\left[K(K-1)(K-2)\indic{K\leq \ell}\right]
	=\sum_{k=\col{1}}^{\infty} [f(k)-f(k-1)]\sum_{l=k}^{\ell}\rho_l
	=\sum_{k=\col{3}}^{\infty} 3(k-1)(k-2)\sum_{l=k}^{\ell}\rho_l.
	\en
We bound this from below by
	\eq
	\expec\left[K(K-1)(K-2)\indic{K\leq \ell}\right]
	\geq \sum_{k=0}^{\sqrt{\ell}} 3(k-1)(k-2)[\rho_{\geq k}-\rho_{\geq \ell}].
	\en
By Lemma \ref{lem-tail-rho}, for $\tau=5$, the contribution due to $\rho_{\geq \ell}$ is at \ch{most}
	\eq
	\ell^{3/2} \rho_{\geq \ell}\leq C_{\rho} \ell^{-3/2}=o(1),
	\en
while the contribution due to $\rho_{\geq k}$ and using $3(k-1)(k-2)\geq k^2$ for every $k\geq 4$, is at least
	\eq
	c_{\rho}\sum_{k=4}^{\sqrt{\ell}} k^{-1}\geq c_{\rho}\int_4^{\sqrt{\ell}+1}\frac{ {\rm d} x}{x}\ch{=} c_{\rho} [\log{(\sqrt{\ell}+1)}-\log{4}],
	\en
which proves the claim by \ch{chosing the constant $c_{3,5}$ correctly.}
\end{proof}

\section{Critical temperature\label{sec-CritTemp}}
In this section we compute the critical temperature.

\begin{proof}[Proof of Theorem~\ref{thm-CritTemp}]
Let $\beta^*=\atanh(1/\nu)$. We first show that if $\beta < \beta^*$, then
\eq
\lim_{B\searrow0} M(\beta,B) = 0,
\en
which implies that $\beta_c\geq \beta^*$. Later, we show that if $\lim_{B\searrow0}M(\beta,B) = 0$ then $\beta \leq \beta^*$, implying that $\beta_c\leq\beta^*$.

\paragraph{Proof of $\beta_c\geq \beta^*$.} Suppose that $\beta<\beta^*$. Then, by the fact that $\tanh x \leq x$ and Wald's identity,
	\eq\label{eq-crittemp1}
	M(\beta,B) = \expec\left[\tanh\left(B+\sum_{i=1}^D\xi(h_i)\right)\right]
	\leq B+\expec[D]\expec[\xi(h)].
	\en
We use the upper bound in Lemma~\ref{lem-boundatanh} to get
	\eq\label{eq-crittemp2}
	\expec[\xi(h)] = \expec[\atanh(\betahat \tanh h)] \leq \betahat \expec[h]
	= \betahat\left(B+\nu\expec[\xi(h)]\right).
	\en
Further, note that
	\eq\label{eq-crittemp3}
	\expec[\xi(h)]= \expec[\atanh(\betahat \tanh h)]\leq\beta,
	\en
because $\tanh h \leq 1$.
Applying inequality~\eqref{eq-crittemp2} $\ell$ times to~\eqref{eq-crittemp1} and subsequently using inequality~\eqref{eq-crittemp3} once gives
	\eq
	M(\beta,B) \leq B + B \betahat\expec[D]\frac{1-(\betahat\nu)^\ell}{1-\betahat\nu}
	+ \beta \expec[D] (\betahat \nu)^\ell.
	\en
Hence,
	\begin{align}
	M(\beta,B) &\leq \limsup_{\ell\rightarrow\infty}
	\left(B + B \betahat\expec[D]\frac{1-(\betahat\nu)^\ell}{1-\betahat\nu}
	+ \beta \expec[D] (\betahat \nu)^\ell \right) \nn\\
	&= B\left(1+\betahat\expec[D]\frac{1}{1-\betahat\nu}\right),
	\end{align}
because $\betahat <\betahat^*= 1/\nu$. Therefore,
	\eq
	\lim_{B\searrow 0} M(\beta,B) \leq \lim_{B\searrow 0} B
	\left(1+\betahat\expec[D]\frac{1}{1-\betahat\nu}\right) = 0.
	\en
This proves the lower bound on $\beta_c$.

\paragraph{Proof of $\beta_c\leq \beta^*$.}
We adapt Lyons' proof in~\cite{Lyo89} for the critical temperature of deterministic trees to the random tree to show that $\ch{\beta_c} \leq \beta^*$. Assume that $\lim_{B\searrow0}M(\beta,B) = 0$. Note that Proposition~\ref{prop-Magnetization} shows that the magnetization $M(\beta,B)$ is equal to the expectation over the random tree $\calT(D,K,\infty)$ of the root magnetization. Hence, if we denote the root of the tree $\calT(D,K,\infty)$ by $\phi$, then it follows from our assumption on $M(\beta,B)$ that, a.s., $\lim_{B\searrow0}\langle\sigma_\phi\rangle=0$.

We therefore condition on the tree $T=\calT(D,K,\infty)$ and if we suppose that $\lim_{B\searrow0} \langle\sigma_\phi\rangle=0$, then also $\lim_{B\searrow0} h(\phi) =0$. Because of~\eqref{eq-recursion}, we must then have, for all $v\in T$,
	\eq
	\lim_{B\searrow0}\lim_{\ell\rightarrow\infty} h^{\ell,+}(v) = 0,
	\en
where we can take $+$ boundary conditions, since the recursion converges to a unique fixed point~\cite[Proposition 1.7]{DomGiaHof10}. Now, fix $0<\beta_0<\beta$ and choose $\ell$ large enough and $B$ small enough such that, for some $\varepsilon=\varepsilon(\beta_0,\beta)>0$ that we choose later,
	\eq
	h^{\ell,+}(v) \leq \varepsilon,
	\en
for all $v\in T$ with $|v|=1$, where $|v|$ denotes the graph distance from $\phi$ to $v$. Note that $h^{\ell,+}(v) =\infty >  \varepsilon$ for $v\in T$ with $|v|=\ell$.

As in~\cite{Lyo89}, we say that $\Pi$ is a {\em cutset} if $\Pi$ is a finite subset of $T\setminus \{\phi\}$ and every path from $\phi$ to infinity intersects $\Pi$ at exactly one vertex $v\in\Pi$. We write $v \leq \Pi$ if every infinite path from $v$ intersects $\Pi$ and write $\sigma < \Pi$ if $\sigma \leq \Pi$ and $\sigma \notin \Pi$. Furthermore, we say that $w\leftarrow v$ if $\{w,v\}$ is an edge in $T$ and $|w|=|v|+1$. Then, since $h^{\ell,+}(v) =\infty >  \varepsilon$ for $v\in T$ with $|v|=\ell$, there is a unique cutset $\Pi$, such that $h^{\ell,+}(v) \leq \varepsilon$ for all $v \leq \Pi$, and for all $v \in \Pi$ there is at least one $w \leftarrow v$ such that $h^{\ell,+}(w) > \varepsilon$.

It follows from the lower bound in Lemma~\ref{lem-boundatanh} that, for $v < \Pi$,
	\eq
	h^{\ell,+}(v) = B+\sum_{w\leftarrow v} \xi(h^{\ell,+}(w))
	\geq \sum_{w\leftarrow v} \betahat h^{\ell,+}(w) - \frac{\betahat h^{\ell,+}(w)^3}{3(1-\betahat^2)}
	\geq \sum_{w\leftarrow v} \betahat h^{\ell,+}(w) \Big(1-\frac{\varepsilon^2}{3(1-\betahat^2)}\Big),
	\en
while, for $v\in\Pi$,
	\eq
	h^{\ell,+}(v) = B+\sum_{w\leftarrow v} \xi( \tanh h(w)) > \xi(\varepsilon).
	\en
If we now choose $\varepsilon>0$ such that
	\eq
	\betahat \Big(1-\frac{\varepsilon^2}{3(1-\betahat^2)^2}\Big) = \betahat_0,
	\en
which is possible because $\beta_0<\beta$, then,
	\eq
	h^{\ell,+}(\phi) \geq \sum_{v\in\Pi} \betahat_0^{|v|} \xi(\varepsilon).
	\en
Since $\xi(\varepsilon)>0$ and  $\lim_{B\searrow0}\lim_{\ell\rightarrow\infty} h^{\ell,+}(\phi) = 0$,
	\eq
	\inf_{\Pi} \sum_{v\in\Pi} \betahat_0^{|v|}=0.
	\en
From \cite[Proposition 6.4]{Lyo90} it follows that $\betahat_0 \leq 1/\nu$. This holds for all $\beta_0<\beta$, so
	\eq
	\beta \leq \atanh(1 / \nu) = \beta^*.
	\en
This proves the upper bound on $\beta_c$, thus concluding the proof.
\end{proof}

We next show that the phase transition at this critical temperature is \emph{continuous:}

\begin{lemma}[Continuous phase transition]\label{lem-hcto0}
It holds that
	\eq
	\lim_{B\searrow 0} \expec[\xi(h(\beta_c,B))] = 0, \qquad {\rm and }
	\qquad \lim_{\beta\searrow \beta_c} \expec[\xi(h(\beta,0^+))]=0.
	\en
\end{lemma}

\begin{proof}
\chs{Note that $\lim_{B\searrow 0} \expec[\xi(h(\beta_c,B))]=c$ exists, because $B\mapsto \expec[\xi(h(\beta_c,B))]$ is non-decreasing and non-negative.}
Assume, by contradiction, that $c>0$.
By the recursion in~\eqref{eq-recursion}, for $B>0$,
	\eq \label{eq-strictineq}
	\expec[\xi(h(\beta,B))] = \expec\Big[\xi\Big(B+ \sum_{i=1}^K\xi(h_i(\beta,B))\Big)\Big]
	\leq 	\xi\left(B+\nu \expec[\xi(h(\beta,B))]\right),
	\en
where the inequality holds because of Jensen's inequality and the concavity of $h\mapsto\xi(h)$. Hence,
	\eq
	c=\lim_{\chs{B\searrow0}}\expec[\xi(h(\beta_c,B))]
	\leq \lim_{\chs{B\searrow0}}\xi\Big(B+\nu \expec[\xi(h(\beta_c,B))]\Big) = \xi(\nu c).
	\en
Since $\xi(x) < \betahat_c x$ for $x>0$ by Lemma~\ref{lem-boundatanh} and using $\betahat_c = 1/\nu$, we obtain
	\eq
	\xi(\nu c) < \betahat_c \nu c = c,
	\en
leading to a contradiction.

\ch{An adaptation of this} argument \ch{shows} \chs{the second statement of the lemma.}
\ch{Again $\beta\mapsto \expec[\xi(h(\beta,0^+))]$ is non-\chs{decreasing and non-negative and we} assume
that $\lim_{\beta\searrow \beta_c} \expec[\xi(h(\beta,0^+))]=c>0$. Then,
	\begin{align}
	c&=\lim_{\beta\searrow \beta_c} \expec[\xi(h(\beta,0^+))]
	=\lim_{\beta\searrow \beta_c} \lim_{B\searrow 0}\expec\Big[\xi\Big(B+ \sum_{i=1}^K\xi(h_i(\beta,B))\Big)\Big]\nn\\
	&\leq \lim_{\beta\searrow \beta_c} \lim_{B\searrow 0}\xi\Big(B+\nu \expec[\xi(h(\beta,B))]\Big)
	=\xi(\nu c),
	\end{align}
leading again to a contradiction when $c>0$.}
%
%For the same reason as above, also the set $\{\xi(h(\beta,0^+))\}_{\beta\geq\beta_c}$ is tight. Therefore, every sequence has a weakly converging subsequence as $\beta\searrow\beta_c$. If it does not hold that $\xi(h(\beta,0^+))\stackrel{\prob}{\rightarrow}0$ for $\beta\searrow\beta_c$ there thus must be a subsequence $\beta_n\rightarrow \beta_c$ for $n\rightarrow\infty$, such that $\xi(h(\beta_n,0^+))\stackrel{d}{\rightarrow}\xi(h^*)\neq0$. Hence, also $\expec[\xi(h^*)]=c$ for some $c>0$. As above, we can bound
%\begin{align}
%c&=\lim_{n\rightarrow\infty}\expec[\xi(h(\beta_n,0^+))] = \lim_{n\rightarrow\infty} \lim_{B\searrow 0} \expec\left[\xi\left(B+\sum_{i=1}^K \xi(h_i(\beta_n,B))\right)\right] \nn\\
%&\leq \lim_{n\rightarrow\infty} \lim_{B\searrow 0} \xi\left(B+\nu\expec[\xi(h(\beta_n,B))]\right) = \lim_{n\rightarrow\infty} \xi\left(\nu \expec[\xi(h(\beta_n,0^+))]\right) \nn\\
%&=\xi(\nu\expec[\xi(h^*)])=\xi(\nu c) < \betahat_c \nu c = c,
%\end{align}
%leading to a contradiction.
\end{proof}

\section{Critical exponents: Magnetization}
\label{sec-CritExpM}
In this section we prove that the critical exponents related to the magnetization, i.e., $\boldsymbol{\beta}$ and $\boldsymbol{\delta}$, take the values stated in Theorem~\ref{thm-CritExp}. The analysis \ch{involves} Taylor expansions performed up to the right order. By these Taylor expansions, higher moments of $\xi(h)$ appear. Therefore, we first bound these higher moments of $\xi(h)$ in terms of its first moment in Section~\ref{sec-MomentsXi}.

In Section~\ref{sec-boundsExi} we use these bounds to give appropriate bounds on $\expec[\xi(h)]$ which finally allow us to compute the critical exponents $\boldsymbol{\beta}$ and $\boldsymbol{\delta}$ in Section~\ref{sec-CritExpBetaDelta}.

\subsection{Bounds on higher moments of $\xi(h)$\label{sec-MomentsXi}}
%From this section onward,
\col{\chs{Throughout} Section \ref{sec-CritExpM}} we assume that $B$ is sufficiently close to zero and $\beta_c < \beta <\beta_c+\varepsilon$ for $\varepsilon$ sufficiently small. We write $c_i,C_i, i\geq1$ for constants that only depend on $\beta$ and moments of $K$, and satisfy
	\eq \label{eq-boundsCi}
	0 < \liminf_{\beta\searrow\beta_c} C_i(\beta) \leq \limsup_{\beta\searrow\beta_c} C_i(\beta) < \infty.
	\en
Here $C_i$ appears in upper bounds, while $c_i$ appears in lower bounds.
Furthermore, we write $e_i, i\geq1$ for error functions that only depend on $\beta, B, \expec[\xi(h)]$ and moments of $K$, and satisfy
	\eq \label{eq-boundsei}
	\limsup_{B\searrow0} e_i(\beta,B) < \infty \qquad {\rm and }
	\qquad \lim_{B\searrow0} e_i(\beta_c,B) =0.
	\en
\ch{Finally, we write $\nu_k=\expec[K(K-1)\cdots (K-k+1)]$ for the $k$th factorial moment of $K$, so that $\nu_1=\nu$.}

\begin{lemma}[Bounds on second moment of $\xi(h)$]
\label{lem-boundxih2}
Let $\beta\geq\beta_c$ and $B>0$. Then,
	\eq\label{eq-boundxih2}
	\expec[\xi(h)^2] \leq \left\{\begin{array}{ll} C_2 \expec[\xi(h)]^2 + B e_2 & \when
	\expec[K^2]<\infty, \\ \\
                                        C_2 \expec[\xi(h)]^2\log\left(1/\expec[\xi(h)]\right) + B e_2
	&\when \tau=4, \\ \\
                                        C_2 \expec[\xi(h)]^{\tau-2} + B e_2
	&\when \tau\in(3,4).\end{array} \right.
	\en
\end{lemma}

\begin{proof}
We first treat the case $\expec[K^2]<\infty$. We use Lemma~\ref{lem-boundatanh} and the recursion in~\eqref{eq-recursion} to obtain
	\begin{align}
	\expec[\xi(h)^2] &\leq \betahat^2 \expec[h^2]
	= \betahat^2 	\expec\Big[\Big(B+\sum_{i=1}^K\xi(h_i)\Big)^2\Big]\nonumber\\
	&= \betahat^2\left(B^2+2B\nu\expec[\xi(h)]+\nu_2\expec[\xi(h)]^2+\nu\expec[\xi(h)^2]\right).
	\end{align}
Since $1-\betahat^2\nu>0$, because $\beta$ is sufficiently close to $\beta_c$ and $\betahat_c=1/\nu<1$, the lemma holds with
	\eq
	C_2 = \frac{\betahat^2 \nu_2}{1-\betahat^2\nu}, \qquad {\rm and } \qquad
	e_2 = \frac{B\betahat^2+2\betahat^2\nu\expec[\xi(h)]}{1-\betahat^2\nu}.
	\en
\rem{RvdH: Note that this inequality is correct with the RIGHT $C_2$!}
It is not hard to see that~\eqref{eq-boundsCi} holds. For $\ch{e_2}$ the first property of \eqref{eq-boundsei}
can also easily be seen. The second property in \eqref{eq-boundsei} follows from Lemma~\ref{lem-hcto0}.

If $\tau\leq4$, then $\expec[K^2]=\infty$ and the above does not work. To analyze this case,
we apply the recursion~\eqref{eq-recursion} and split the expectation over $K$ in small and large degrees:
	\eq\label{eq-exih2}
	\expec[\xi(h)^2] =
	\expec\Big[\xi\Big(B+\sum_{i=1}^K \xi(h_i)\Big)^2 \ind_{\{K\leq \ell\}}\Big]
	+\expec\Big[\xi\Big(B+\sum_{i=1}^K \xi(h_i)\Big)^2\ind_{\{K > \ell\}}\Big].
	\en
We use Lemma~\ref{lem-boundatanh} to bound the first term as follows:
	\begin{align}
	\expec\Big[\xi\Big(B+\sum_{i=1}^K \xi(h_i)\Big)^2 \ind_{\{K\leq \ell\}}\Big]
	&\leq \betahat^2 \expec\Big[\Big(B+\sum_{i=1}^K \xi(h_i)\Big)^2 \ind_{\{K\leq \ell\}}\Big] \\
	&\leq \betahat^2\left(B^2+2B\nu B\expec[\xi(h)]
	+ \expec[K^2\ind_{\{K\leq \ell\}}]\expec[\xi(h)]^2+\nu\expec[\xi(h)^2] \right).\nn
	\end{align}
For $\tau\in(3,4)$,
	\eq
	\expec[K^2\ind_{\{K\leq \ell\}}] \leq C_{2,\tau}\ell^{4-\tau},
	\en
by Lemma~\ref{lem-truncmoment}, while for $\tau=4$,
	\eq
	\expec[K^2\ind_{\{K\leq \ell\}}] \leq  C_{2,4}\log \ell.
	\en

To bound the second sum in~\eqref{eq-exih2}, note that $\xi(h)\leq \beta$. Hence,
	\eq
	\expec\Big(\xi\Big(B+\sum_{i=1}^K \xi(h_i)\Big)^2\ind_{\{K > \ell\}}\Big]
	\leq \beta^2 	\expec[\ind_{\{K > \ell\}}] \leq C_{0,\tau}\beta^2\ell^{2-\tau}.
	\en
The optimal bound (up to a constant) can be achieved by choosing $\ell$ such that $\ell^{4-\tau}\expec[\xi(h)]^2$ and $\ell^{2-\tau}$ are of the some order of magnitude. Hence, we choose $\ell=1/\expec[\xi(h)]$. Combining the two upper bounds then gives the desired result with
	\eq
	\ch{C_2=\frac{1}{1-\betahat^2\nu}\left(C_{2,\tau}\betahat^2+C_{0,\tau}\beta^2\right),}
	\en
where we \ch{have} also used that $\expec[\xi(h)]^2\leq\expec[\xi(h)]^2\log(1/\expec[\xi(h)])$, and
	\eq
	e_2 = \frac{B\betahat^2+2\betahat^2\nu\expec[\xi(h)]}{1-\betahat^2\nu}.
	\en
\end{proof}

We next derive upper \ch{bounds} on the third moment of $\xi(h)$:
\begin{lemma}[Bounds on third moment of $\xi(h)$]\label{lem-boundxih3}
Let $\beta\geq\beta_c$ and $B>0$. Then,
\eq
\expec[\xi(h)^3] \leq \left\{\begin{array}{ll} C_3 \expec[\xi(h)]^3 + B e_3 &\when \expec[K^3]<\infty, \\ \\
                                               C_3 \expec[\xi(h)]^3\log\left(1/\expec[\xi(h)]\right) + B e_3 &\when \tau=5, \\ \\
                                               C_3 \expec[\xi(h)]^{\tau-2} + B e_3 &\when \tau\in(3,5).\end{array} \right.
\en
\end{lemma}

\begin{proof}
For $\expec[K^3]<\infty$ we bound, in a similar way as in Lemma~\ref{lem-boundxih2},
	\begin{align}
	\expec[\xi(h)^3] \leq \betahat^3&\Bigg(B^3 + 3 B^2\nu\expec[\xi(h)]
	+3B\nu_2\expec[\xi(h)]^2+3B\nu\expec[\xi(h)^2]\\
	&+\nu_3\expec[\xi(h)]^3+3\nu_2\expec[\xi(h)]\expec[\xi(h)^2]
	+\betahat^3\nu\expec[\xi(h)^3]\Bigg).\nn
	\end{align}
Using~\eqref{eq-boundxih2}, we indeed get the bound
	\eq
	\expec[\xi(h)^3] \leq C_3 \expec[\xi(h)]^3 + B e_3,
	\en
where
	\eq
	C_3 = \frac{\betahat^3}{1-\betahat^3 \nu} \left(\nu_3 + 3 \nu_2 C_2\right),
	\en
\rem{RvdH: Note that this is the RIGHT $C_3$!}
and
	\eq
	e_3 = \frac{\betahat^3}{1-\betahat^3\nu}\left\{B^2
	+3B\nu e_2+3\left(B\nu+\nu_2e_2\right)\expec[\xi(h)]
	+3\left(\nu_2+\nu C_2\right)\expec[\xi(h)]^2\right\}.
	\en

For $\tau\in(3,5]$, we use the recursion~\eqref{eq-recursion}, make the expectation over $K$ explicit and split in small and large values of $K$ to obtain
	\eq\label{eq-exih3}
	\ch{\expec[\xi(h)^3] =
	\expec\Big[\xi\Big(B+\sum_{i=1}^K\xi(h_i)\Big)^3\indic{K\leq \lfloor1/\expec[\xi(h)]\rfloor}\Big] +
	\expec\Big[\xi\Big(B+\sum_{i=1}^K\xi(h_i)\Big)^3\indic{K>\lfloor1/\expec[\xi(h)]\rfloor}\Big].}
	\en
We bound the first sum from above by
	\ch{\begin{align}
	&\betahat^3 \expec\Big[\Big(B+\sum_{i=1}^k\xi(h_i)\Big)^3\indic{K\leq \lfloor1/\expec[\xi(h)]\rfloor}\Big] \nn\\
	&\qquad= \betahat^3 \Big(B^3 + 3B^2 \expec[K \indic{K\leq \lfloor1/\expec[\xi(h)]\rfloor}]\expec[\xi(h)]+
	3B\expec[K(K-1)\indic{K\leq \lfloor1/\expec[\xi(h)]\rfloor}]\expec[\xi(h)^2]\nn\\
	&\qquad\qquad+3B\expec[K \indic{K\leq \lfloor1/\expec[\xi(h)]\rfloor}]
	\expec[\xi(h)^2]+\expec[K(K-1)(K-2) \indic{K\leq \lfloor1/\expec[\xi(h)]\rfloor}]
	\expec[\xi(h)]^3
	\nn\\
	&\qquad\qquad+3\expec[K(K-1)\indic{K\leq \lfloor1/\expec[\xi(h)]\rfloor}]\expec[\xi(h)]\expec[\xi(h)^2]+\expec[K\indic{K\leq \lfloor1/\expec[\xi(h)]\rfloor}]\expec[\xi(h)^3]\Big).\nn
\end{align}}
By Lemma~\ref{lem-truncmoment}, for $\tau\in(3,5)$,
	\ch{\eq
	\expec[K^3\indic{K\leq \lfloor1/\expec[\xi(h)]\rfloor}]\leq C_{3,\tau}\expec[\xi(h)]^{\tau-5},
	\en}
while, for $\tau=5$,
	\ch{\eq
	\expec[K^3\indic{K\leq \lfloor1/\expec[\xi(h)]\rfloor}]
	\leq C_{3,5}\left(1+\log\left(1/\expec[\xi(h)]\right)\right).
	\en}
Similarly, by Lemma~\ref{lem-truncmoment}, for $\tau\in(3,4)$,
	\ch{\eq
	\expec[K^2\indic{K\leq \lfloor1/\expec[\xi(h)]\rfloor}]\leq C_{2,\tau}\expec[\xi(h)]^{\tau-4},
	\en}
while, for $\tau=4$,
	\ch{\eq
	\expec[K^2\indic{K\leq \lfloor1/\expec[\xi(h)]\rfloor}]
	\leq C_{2,4}\left(1+\log\left(1/\expec[\xi(h)]\right)\right).
	\en}

For the other terms we can replace the upper bound in the sum by infinity and use the upper bound on $\expec[\xi(h)^2]$ of Lemma~\ref{lem-boundxih2}. For the second sum in~\eqref{eq-exih3} we bound $\xi(x)\leq\beta$, so that this sum is bounded from above by $C_{0,\tau}\expec[\xi(h)]^{\tau-2}$.
Combining these bounds gives the desired result.
\end{proof}

%\todo{Update from here!}
\subsection{Bounds on first moment of $\xi(h)$ \label{sec-boundsExi}}
\begin{proposition}[Upper bound on first moment of $\xi(h)$]
\label{prop-UpperExi}
Let $\beta\geq\beta_c$ and $B>0$. Then, there exists a $C_1>0$ such that
	\eq\label{eq-UpperExi}
	\expec[\xi(h)] \leq \beta B + \betahat \nu \expec[\xi(h)] - C_1 \expec[\xi(h)]^{\boldsymbol{\delta}},
	\en
where
	\eq
	\boldsymbol{\delta}
	=\left\{\begin{array}{ll} 3 &\when \expec[K^3]<\infty,\\ \\ \tau-2 &\when \tau\in(3,5].
	\end{array}\right.
	\en
For $\tau=5$,
	\eq\label{eq-UpperExi-tau5}
	\expec[\xi(h)] \leq \beta B + \betahat \nu \expec[\xi(h)]
	- C_1 \expec[\xi(h)]^{3}\log\left(1/\expec[\xi(h)]\right).
	\en
\end{proposition}
%Note that this upper bound is also true for the boundary cases $\tau=4$ and $\tau=5$.

\begin{proof}
We first use recursion~\eqref{eq-recursion} and rewrite it as
	\eq\label{eq-split}
	\expec[\xi(h)] = \expec\Big[\xi\Big(B+\sum_{i=1}^K \xi(h_i)\Big)\Big]
	= \betahat B + \betahat \nu \expec[\xi(h)] + T_1+T_2,
	\en
where
	\eq
	T_1= \expec\Big[\xi\Big(B+K\expec[\xi(h)]\Big)-\betahat\left(B + K \expec[\xi(h)]\right)\Big],
	\en
and
	\eq
	T_2 = \expec\Big[\xi\Big(B+\sum_{i=1}^K \xi(h_i)\Big)-\xi\left(B+K\expec[\xi(h)]\right)\Big].
	\en
Here, $T_1$ can be seen as the error of a first order Taylor series approximation of $\xi\left(B+K\expec[\xi(h)]\right)$ around $0$, whereas $T_2$ is the error made by replacing $\xi(h_i)$ by its expected value in the sum. By concavity of $x\mapsto \xi(x)$, both random variables in the expectation of $T_1$ and $T_2$ are non-positive. In particular, $T_2\leq 0$, which is enough for our purposes. We next bound $T_1$ in the cases where $\expec[K^3]<\infty$, $\tau\in(3,5)$ and $\tau=5$ separately.

\paragraph{Bound on $T_1$ when $\expec[K^3]<\infty$.}
To bound $T_1$ for $\expec[K^3]<\infty$ we use that, a.s.,
	\eq
	\xi\left(B+K\expec[\xi(h)]\right)-\betahat\left(B + K\expec[\xi(h)]\right)\leq0,
	\en
which follows from Lemma~\ref{lem-boundatanh}. Hence,
	\eq
	T_1 \leq \expec\left[\left(\xi\left(B+K\expec[\xi(h)]\right)
	-\betahat\left(B + K\expec[\xi(h)]\right)\right)
	\ind_{\{B + K \expec[\xi(h)]\leq \atanh\frac12\}}\right].
	\en
Since $\xi''(0)=0$, it follows from Taylor's theorem that, a.s.,
	\eq
	\xi\left(B+K\expec[\xi(h)]\right)-\betahat\left(B + K\expec[\xi(h)]\right)
	=\frac{\xi'''(\zeta)}{6}\left(B + K\expec[\xi(h)]\right)^3,
	\en
for some $\zeta\in(0,B + K\expec[\xi(h)])$. If $B + K \expec[\xi(h)]\leq \atanh\frac12$, then
	\eq
	\xi'''(\zeta)= -\frac{2 \betahat (1-\betahat^2) (1-\tanh^2 \zeta)}
	{(1-\betahat^2\tanh^2\zeta)^3}\left(1-3(1-\betahat^2)\tanh^2\zeta-\betahat^2\tanh^4\zeta\right)
	\leq-\frac{9}{32}\betahat(1-\betahat^2).
	\en
Hence,
	\begin{align}
	T_1&\leq-\frac{3}{64}\betahat(1-\betahat^2)\expec\left[\left(B + K\expec[\xi(h)]\right)^3
	\ind_{\{B + K \expec[\xi(h)]\leq \atanh\frac12\}}\right]\nn\\
	&\leq-\frac{3}{64}\betahat(1-\betahat^2)\expec[K^3\ind_{\{K \expec[\xi(h)]
	\leq \atanh\frac12-B\}}]\expec[\xi(h)]^3.
	\end{align}

\paragraph{Bound on $T_1$ \ch{when} $\tau\in(3,5]$.}
For $\tau\in(3,5]$, we make the expectation over $K$ explicit:
	\eq
	T_1=\sum_{k=0}^\infty \rho_k\left(\xi\left(B+k\expec[\xi(h)]\right)
	-\betahat\left(B + k\expec[\xi(h)]\right)\right),
	\en
where it should be noted that all terms in this sum are negative
because of Lemma~\ref{lem-boundatanh}. Define $f(k)=\xi\left(B+k\expec[\xi(h)]\right)
-\betahat\left(B + k\expec[\xi(h)]\right)$ and note that $f(k)$ is non-increasing.
We use \eqref{partial-summation} and Lemma \ref{lem-tail-rho} to rewrite
	\eq
	T_1=\sum_{k=0}^\infty f(k)\rho_k
	=\ch{f(0)+}\sum_{k\geq 1} [f(k)-f(k-1)]\rho_{\geq k}
	\leq \ch{f(0)+}c_{\rho}\sum_{k\geq 1} [f(k)-f(k-1)] (k+1)^{-(\tau-2)}.
	\en
Then, use \eqref{partial-summation}  in reverse to rewrite this as
	\eq
	T_1
	\leq \ch{f(0)+}c_{\rho}\sum_{k\geq 0} f(k) [k^{-(\tau-2)}-(k+1)^{-(\tau-2)}]
	\leq \ch{f(0)(1-c_{\rho}\sum_{k\geq 1} k^{-\tau})+}(\tau-1)c_{\rho} \sum_{k\geq 0} f(k) (k+1)^{-(\tau-1)}.
	\en
Hence, with $\ch{e=f(0)(1-c_{\rho}\sum_{k\geq 1} k^{-\tau})/B}$,
	\begin{align}
	T_1&\leq \ch{eB+} (\tau-1)c_{\rho} \left(\expec[\xi(h)]\right)^{\tau-1}
	\sum_{k=0}^{\infty} ((k+1)\expec[\xi(h)])^{-(\tau-1)}
	\left(\xi\left(B+k\expec[\xi(h)]\right)-\betahat\left(B + k\expec[\xi(h)]\right)\right)\nn\\
	&\leq \ch{eB+} (\tau-1)c_{\rho} \left(\expec[\xi(h)]\right)^{\tau-1}
	 \sum_{k=a/\expec[\xi(h)]}^{b/\expec[\xi(h)]} (k\expec[\xi(h)])^{-(\tau-1)}
	\left(\xi\left(B+k\expec[\xi(h)]\right)-\betahat\left(B + k\expec[\xi(h)]\right)\right),
\end{align}
where we choose $a$ and $b$ such that $0<a<b<\infty$. We use dominated convergence on the above sum.
The summands are  uniformly bounded, and $\expec[\xi(h)]\rightarrow0$ for both limits of interest.
Further, when $k\expec[\xi(h)]=y$, the summand converges pointwise to
$y^{-(\tau-1)} \left(\xi\left(B+y\right)-\betahat\left(B + y\right)\right)$.
Hence, we can write the sum above as
	\eq
	\expec[\xi(h)]^{-1} \left(\int_{a}^{b} y^{-(\tau-1)}
	\left(\xi\left(B+y\right)-\betahat\left(B + y\right)\right)\dint y +o(1)\right),
	\en
where $o(1)$ is a function tending to zero for both limits of interest \cite[216 A]{Ito93}. The integrand is uniformly bounded for $y\in[a,b]$ and hence we can bound the integral from above by a (negative) constant $-I$ for $B$ sufficiently small and $\beta$ sufficiently close to $\beta_c$. Hence,
	\eq
	\expec[\xi(h)] \leq \betahat B + \betahat \nu \expec[\xi(h)]
	- (\tau-1)c_{\rho}  I\expec[\xi(h)]^{\tau-2}.
	\en

%It remains to bound $T_2$. This can also be seen as a Taylor expansion of $\xi(B+\sum_{i=1}^K\xi(h_i))$ around $B+K\expec[\xi(h)]$. Note that, a.s.,
%\eq
%\expec\left[\xi'(B+K\expec[\xi(h)]) \left(\sum_{i=1}^K\xi(h_i)-K\expec[\xi(h)]\right)\bigg| K \right] = 0,
%\en
%and hence also the expectation over $K$ of the above equals $0$. Thus,
%\eq\label{eq-T2Taylor}
%T_2=\expec\left[\frac{\xi''(\zeta)}{2}\left(\sum_{i=1}^K\xi(h_i) - K\expec[\xi(h)]\right)^2\right],
%\en
%for some $\zeta\in\left(B+\sum_{i=1}^K\xi(h_i), B+K\expec[\xi(h)]\right)$. Since $\xi''(\zeta)\leq 0$, it follows that $T_2\leq0$.
%
%Combining the bounds on $T_1$ and $T_2$ gives the desired bound on $\expec[\xi(h)]$ when $\tau\in (3,5]$.

\paragraph{Logarithmic corrections in the bound for $\tau=5$.}
We complete the proof by identifying the logarithmic correction for $\tau=5$. Since the
random variable in the expectation in $T_1$ is non-positive, we can bound
	\eq
	T_1\leq
	\expec\left[\xi\left(B+K\expec[\xi(h)]\right)-\betahat\left(B + K \expec[\xi(h)]\right)
	\indic{K\leq \vep /\expec[\xi(h)]}\right].
	\en
Taylor expansion $h\mapsto \xi(h)$ to third order, using that $\xi(0)=\xi''(0)=0$,
while the linear term cancels, leads to
	\eq
	T_1\leq
	\expec\left[\frac{\xi'''(\zeta)}{6} \left(B+K\expec[\xi(h)]\right)^3
	\indic{K\leq \vep /\expec[\xi(h)]}\right],
	\en
for some $\zeta\in (0,K\expec[\xi(h)])$. On the event that $K\leq \vep /\expec[\xi(h)]$,
we thus have that $\zeta\in (0,\vep)$, and $\xi'''(\zeta)\geq \ch{c_{\vep}\equiv\inf_{x\in (0,\vep)} \xi'''(x)}$
when $\vep$ is sufficiently small. Thus,
	\begin{align}
	T_1&\leq
	\frac{\ch{c_{\vep}}}{6}\expec\left[\left(B+K\expec[\xi(h)]\right)^3
	\indic{K\leq \vep /\expec[\xi(h)]}\right]\\
	&\leq \frac{\ch{c_{\vep}}}{6} \expec[\xi(h)]^3 \expec\left[K(K-1)(K-2)\indic{K\leq \vep/\expec[\xi(h)]}\right].\nn
	\end{align}
When $\tau=5$, by Lemma \ref{lem-truncmoment}, $\expec\left[K(K-1)(K-2)\indic{K\leq \ell}\right]\geq c_{3,5}\log{\ell}$,
which completes the proof.
\rem{For $\expec[K^3]<\infty$, the above argument yields the correct upper bound
on $T_1$ when $\vep\searrow 0$!  In turn, the correct upper bound on
$T_1$ will yield the upper lower bound on $\expec[\xi(h)]$, which is the wrong bound...}
\end{proof}

\begin{proposition}[Lower bound on first moment of $\xi(h)$]\label{prop-LowerExi} Let $\beta\geq\beta_c$ and $B>0$. Then, there exists a constant $C_2>0$ such that
	\eq
	\label{eq-LowerExi}
	\expec[\xi(h)] \geq
	\beta B + \betahat \nu \expec[\xi(h)] - c_1 \expec[\xi(h)]^{\boldsymbol{\delta}} - B e_1,
	\en
where
	\eq
	\boldsymbol{\delta}=\left\{\begin{array}{ll} 3 &\when \expec[K^3]<\infty,\\ \\ \tau-2 &
	\when\tau\in(3,5).\end{array}\right.
	\en
For $\tau=5$,
	\eq
	\label{eq-LowerExi-tau5}
	\expec[\xi(h)] \geq \beta B + \betahat \nu \expec[\xi(h)]
	- C_2 \expec[\xi(h)]^3 \log(1/\expec[\xi(h)]) - B e_1.
	\en
\end{proposition}

\begin{proof}
We again use the split in~\eqref{eq-split} and we bound $T_1$ and $T_2$.

\paragraph{The lower bound on $T_1$.} For $\expec[K^3]<\infty$, we use the lower bound of Lemma~\ref{lem-boundatanh} to get
	\eq\label{eq-t1k3finite}
	T_1 \geq - \frac{\betahat}{3(1-\betahat^2)} \expec\left[(B+K\expec[\xi(h)])^3\right].
	\en
By expanding, this can be rewritten as
	\eq
	T_1 \geq - \frac{\betahat}{3(1-\betahat^2)} \expec[K^3]\expec[\xi(h)]^3-B e_4.
	\en

For $\tau\in(3,5]$, we first split $T_1$ in a small $K$ and a large $K$ part. For this, write
	\eq
	t_1(k) = \xi\left(B+k\expec[\xi(h)]\right)-\betahat\left(B + k\expec[\xi(h)]\right).
	\en
Then,
	\eq
	T_1 = \expec[t_1(K)]=\expec\left[t_1(K) \ind_{\{K \leq \varepsilon/\expec[\xi(h)]\}}\right]
	+\expec\left[t_1(K) \ind_{\{K > \varepsilon/\expec[\xi(h)]\}}\right].
	\en
To bound the first term, we again use~\eqref{eq-t1k3finite}:
	\eq
	\expec\left[t_1(K) \ind_{\{K \leq \varepsilon/\expec[\xi(h)]\}}\right]
	\geq - \frac{\betahat}{3(1-\betahat^2)} \expec\left[(B+K\expec[\xi(h)])^3
	\ind_{\{K \leq \varepsilon/\expec[\xi(h)]\}}\right].
	\en
It is easy to see that the terms $B^3\expec\left[\ind_{\{K > \varepsilon/\expec[\xi(h)]\}}\right]$ and $3B^2\expec[\xi(h)]\expec\left[K\ind_{\{K \leq \varepsilon/\expec[\xi(h)]\}}\right]$ that we get by expanding the above are of the form $B e$. To bound the other two terms, we use Lemma~\ref{lem-truncmoment} to obtain,
\ch{for $\vep\leq 1,$}
\eq
3B \expec[\xi(h)]^2\expec\left[K^2\ind_{\{K\leq \varepsilon/\expec[\xi(h)]\}}\right] \leq \left\{\begin{array}{ll} 3B \expec[\xi(h)]^2\expec\left[K^2\right]&\when \tau\in(4,5], \\ \\
3B C_{2,4}\expec[\xi(h)]^2 \log(1/\expec[\xi(h)])&\when \tau=4,\\ \\
3B C_{2,\tau}\expec[\xi(h)]^{\tau-2}&\when \tau\in(3,4),\end{array}\right.
\en
which are all of the form $B e$, and
\eq
\expec\left[K^3\ind_{\{K \leq \varepsilon/\expec[\xi(h)]\}}\right] \expec[\xi(h)]^3 \leq \left\{\begin{array}{ll} C_{3,5}\expec[\xi(h)]^3\log(1/\expec[\xi(h)])&\when \tau=5, \\ \\
C_{3,\tau}\expec[\xi(h)]^{\tau-2}&\when \tau\in(3,5).\end{array}\right.
\en
To bound $T_1$ for large $K$, we observe that
\eq
\expec\left[t_1(K) \ind_{\{K > \varepsilon/\expec[\xi(h)]\}}\right] \geq -\betahat B \expec[\ind_{\{K > \varepsilon/\expec[\xi(h)]\}}]-\betahat\expec[\xi(h)]\expec[K\ind_{\{K > \varepsilon/\expec[\xi(h)]\}}].
\en
Applying Lemma~\ref{lem-truncmoment} now gives, for $\tau\in(3,5]$
	\begin{align}
	\expec\left[t_1(K) \ind_{\{K > \varepsilon/\expec[\xi(h)]\}}\right]
	&\geq -\betahat B C_{0,\tau}\expec[\xi(h)]^{\tau-2} - \betahat C_{1,\tau}\expec[\xi(h)]^{\tau-2} \nn\\
	&= -C_4 \expec[\xi(h)]^{\tau-2}-B e_4.
	\end{align}

\paragraph{The lower bound on $T_2$.}
To bound $T_2$,  we split in a small and a large $K$ contribution:
	\eq
	T_2 = \expec[t_2(K) \ind_{\{K \leq \varepsilon/\expec[\xi(h)]\}}]+\expec[t_2(K)
	\ind_{\{K > \varepsilon/\expec[\xi(h)]\}}]\equiv T_2^{\sss \leq}+T_2^{\sss >},
	\en
where
	\eq
	t_2(k)=\xi\left(B+\sum_{i=1}^k \xi(h_i)\right)-\xi\left(B+k\expec[\xi(h)]\right).
	\en
To bound $T_2^{\sss >}$, we note that
	\eq
	t_2(k) \geq -\beta,
	\en
so that
	\eq
	T_2^{\sss >}
	\geq -\beta \expec[\ind_{\{K > \varepsilon/\expec[\xi(h)]\}}] \geq -C_5 \expec[\xi(h)]^{\ch{(\tau-2)\wedge 3}},
	\en
where we have used Lemma~\ref{lem-truncmoment} in the last inequality \ch{and the Markov inequality when $\expec[K^3]<\infty$}.

To bound $T_2^{\sss \leq}$, we start from
	\eq\label{eq-T2Taylor}
	 T_2^{\sss \leq}=\expec\left[\frac{\xi''(\zeta)}{2}\left(\sum_{i=1}^K\xi(h_i) - K\expec[\xi(h)]\right)^2\ind_{\{K \leq \varepsilon/\expec[\xi(h)]\}}\right],
	\en
\ch{We use that, for some $\zeta$ in between $B+\sum_{i=1}^K\xi(h_i)$ and $B+K\expec[\xi(h)]$,}
	\eq
	\xi''(\zeta) \geq -\frac{2\betahat}{1-\betahat^2} \Big(B+\sum_{i=1}^K\xi(h_i)+K\expec[\xi(h)]\Big).
	\en
to obtain
	\begin{align}
	T_2^{\sss \leq}
	& \geq -\frac{\betahat}{(1-\betahat^2)^2} \expec\Big[\Big(B+\sum_{i=1}^K\xi(h_i)
	+ K\expec[\xi(h)]\Big)\Big(\sum_{i=1}^K\xi(h_i) - K\expec[\xi(h)]\Big)^2
	\ind_{\{K \leq \varepsilon/\expec[\xi(h)]\}}\Big]\nn\\
	&\geq -\frac{\betahat}{(1-\betahat^2)^2} \Big(B \nu \expec\left[\left(\xi(h)-\expec[\xi(h)]\right)^2\right]
	+ \ch{\expec[K\ind_{\{K \leq \varepsilon/\expec[\xi(h)]\}}]} \expec\left[\left(\xi(h)-\expec[\xi(h)]\right)^3\right]\nn\\
	&\qquad+ 2\ch{\expec[K^2\ind_{\{K \leq \varepsilon/\expec[\xi(h)]\}}]}\expec[\xi(h)]\expec\left[\left(\xi(h)-\expec[\xi(h)]\right)^2\right]\Big)\nonumber\\
	&\geq -\frac{\betahat}{(1-\betahat^2)^2} \Big(B \nu \expec[\xi(h)^2]
	+2\ch{\expec[K^2\ind_{\{K \leq \varepsilon/\expec[\xi(h)]\}}]}\expec[\xi(h)]\expec[\xi(h)^2]+\nu\expec[\xi(h)^3]\Big).
	\end{align}
Using the bounds of Lemmas~\ref{lem-boundxih2} and~\ref{lem-boundxih3} we get,
	\eq
	T_2^{\sss \leq}\geq
	\left\{\begin{array}{ll} -\frac{\betahat}{(1-\betahat^2)^2}\left(2\expec[K^2]C_2+C_3\nu\right)\expec[\xi(h)]^3-B e_5
	&\when \expec[K^3]<\infty,\\ \\
	-\frac{\betahat}{(1-\betahat^2)^2}\left(2\expec[K^2]C_2+C_3\nu\right)\expec[\xi(h)]^{3}\log(1/\expec[\xi(h)])-B e_5
	&\when \tau=5,\\ \\
	-\frac{\betahat}{(1-\betahat^2)^2}\left(C'_{2,\tau}+C_3\nu\right)\expec[\xi(h)]^{\tau-2}-B e_5
	&\when \tau\in(3,5),\end{array}\right.
	\en
where \ch{$C'_{2,\tau}=\expec[K^2]C_2$ for $\tau\in (4,5)$ and $C'_{2,\tau}=C_2$ for $\tau\in(3,4]$. Here,
we have also used that (a) $\expec[\xi(h)]^3\leq\expec[\xi(h)]^{3}\log(1/\expec[\xi(h)])$ for $\tau=5$;
(b) $\expec[\xi(h)]^3\leq\expec[\xi(h)]^{\tau-2}$ for $\tau\in (4,5]$; and (c) ${\expec[K^2\ind_{\{K \leq \varepsilon/\expec[\xi(h)]\}}]}\expec[\xi(h)]
\leq \vep \nu\leq \nu$ for $\tau\in (3,4]$.}
%
%
%To bound the small $K$ contribution to $T_2$, we do the same as for $\tau>4$, but with the indicator and then use Lemma~\ref{lem-truncmoment} to get the desired bound.
%\todo{Write this out.}
%
%
%For $\tau\in(3,4)$ we can alternatively use the bound
%\eq
%\xi''(\zeta) \geq -\frac{2\betahat}{(1-\betahat^2)^2},
%\en
%so that
%\eq
%T_2 \geq -\frac{\betahat}{(1-\betahat^2)^2} \nu \expec[\xi(h)^2].
%\en
%By Lemma~\ref{lem-boundxih2} the above is at least
%\eq
%-\frac{\betahat}{(1-\betahat^2)^2} \nu \left(C_2\expec[\xi(h)]^{\tau-2} - B e_2\right).
%\en
Combining the bounds on $T_1$ and $T_2$ gives the desired lower bound on $\expec[\xi(h)]$.
\end{proof}

\subsection{Critical exponents $\boldsymbol{\beta}$ and $\boldsymbol{\delta}$\label{sec-CritExpBetaDelta}}
It remains to show that the bounds on $\expec[\xi(h)]$ give us the desired result:
\begin{theorem}[Values of $\boldsymbol{\beta}$ and $\boldsymbol{\delta}$]
The critical exponent $\boldsymbol{\beta}$ equals
	\eq
	\boldsymbol{\beta}=\left\{\begin{array}{ll} 1/2 &\when \expec[K^3]<\infty,\\
	1/(\tau-3) &\when \tau\in(3,5),\end{array}\right.
	\en
and the critical exponent $\boldsymbol{\delta}$ equals
	\eq
	\boldsymbol{\delta}=\left\{\begin{array}{ll} 3 &\when \expec[K^3]<\infty,\\
	\tau-2 &\when \tau\in(3,5).\end{array}\right.
	\en
For $\tau=5$,
	\eq
	M(\beta,0^+) \asymp \Big(\frac{\beta-\beta_c}{\log{(1/(\beta-\beta_c))}}\Big)^{1/2} \quad {\rm for\ } \beta \searrow \beta_c,
	\qquad
	M(\beta_c, B) \asymp \Big(\frac{B}{\log(1/B)}\Big)^{1/3} \quad {\rm for\ } B \searrow 0.
	\en
\end{theorem}

\begin{proof} We prove the upper and the lower bounds separately, starting with the upper bound.
\paragraph{The upper bounds on the magnetization.}
We start by bounding the magnetization from above:
	\eq
	M(\beta,B)=\expec\left[\tanh\left(B+\sum_{i=1}^D\xi(h_i)\right)\right] \leq B+\expec[D]\expec[\xi(h)].
	\en
We first perform the analysis for $\boldsymbol{\beta}$. Taking the limit $B\searrow0$ in \eqref{eq-UpperExi}
in Proposition \ref{prop-UpperExi} yields
	\eq
	\expec[\xi(h_0)] \leq \betahat \nu \expec[\xi(h_0)] - C_1 \expec[\xi(h_0)]^{\boldsymbol{\delta}},
	\en
where $h_0=h(\beta,0^+)$. For $\beta>\beta_c$, by definition, $\expec[\xi(h_0)]>0$ and thus we can divide through by $\expec[\xi(h_0)]$ to obtain
	\eq
	\expec[\xi(h_0)]^{\boldsymbol{\delta}-1} \leq \frac{\betahat \nu-1}{C_1}.
	\en
By Taylor's theorem,
	\eq
	\label{eq-Taylorbetahatup}
	\betahat \nu-1 \leq \nu(1-\betahat_c^2)(\beta-\beta_c).
	\en
Hence,
	\eq
	\label{eq-Exih0}
	\expec[\xi(h_0)] \leq \left(\frac{\nu(1-\betahat_c^2)}{C_1}\right)^{1/(\boldsymbol{\delta}-1)}
	(\beta-\beta_c)^{1/(\boldsymbol{\delta}-1)}.
	\en
Using that $\boldsymbol{\beta}=1/(\boldsymbol{\delta}-1)$,
	\eq
	M(\beta,0^+)\leq \expec[D]\left(\frac{\nu(1-\betahat_c^2)}{C_1}\right)^{\boldsymbol{\beta}}(\beta-\beta_c)^{\boldsymbol{\beta}},
	\en
from which it easily follows that
	\eq
	\label{magnetization-beta-UB}
	\limsup_{\beta\searrow\beta_c} \frac{M(\beta,0^+)}{(\beta-\beta_c)^{\boldsymbol{\beta}}}<\infty.
	\en
We complete the analysis for $\boldsymbol{\beta}$ by analyzing $\tau=5$.
Since \eqref{eq-UpperExi} also applies to $\tau=5$,  \eqref{magnetization-beta-UB}
holds as well. We now improve upon this using \eqref{eq-UpperExi-tau5}
in Proposition \ref{prop-UpperExi}, which yields in a similar way as above that
	\eq
	\expec[\xi(h_0)]^{2} \leq \frac{\betahat \nu-1}{C_1\log(1/\expec[\xi(h_0)])}.
	\en
Since $x\mapsto 1/\log(1/x)$ is increasing on $(0,1)$ and $\expec[\xi(h_0)]\leq C(\beta-\beta_c)^{1/2}$
for some $C>0$, we immediately obtain that
	\eq
	\expec[\xi(h_0)]^{2} \leq \frac{\betahat \nu-1}{C_1\log(1/\expec[\xi(h_0)])}
	\leq \frac{\betahat \nu-1}{C_1\log(1/[C(\beta-\beta_c)^{1/2}])}.
	\en
Taking the limit of $\beta \searrow\beta_c$ as above then completes the proof.

We continue with the analysis for $\boldsymbol{\delta}$.
Setting $\beta=\beta_c$ in~\eqref{eq-UpperExi} and rewriting gives
	\eq
	\expec[\xi(h_c)]\leq\left(\frac{\betahat_c}{C_1}\right)^{1/\boldsymbol{\delta}} B^{1/\boldsymbol{\delta}},
	\en
with $h_c=h(\beta_c,B)$. Hence,
	\eq
	M(\beta_c,B)\leq B+\expec[D]\left(\frac{\betahat_c}{C_1}\right)^{1/\boldsymbol{\delta}} B^{1/\boldsymbol{\delta}},
	\en
so that, using $1/\boldsymbol{\delta}<1$,
	\eq
	\limsup_{B\searrow0} \frac{M(\beta_c,B)}{B^{1/\boldsymbol{\delta}}}<\infty.
	\en
The analysis for $\boldsymbol{\delta}$ for $\tau=5$ can be performed in an identical way as for
$\boldsymbol{\beta}$.

\paragraph{The lower bounds on the magnetization.}
For the lower bound on the magnetization we use that
	\eq
	\frac{\dint^2}{\dint x^2} \tanh x = -2\tanh x (1-\tanh^2 x) \geq -2,
	\en
so that
	\eq
	\tanh x \geq x-x^2.
	\en
Hence,
	\begin{align}
	M(\beta,B) &\geq B + \expec[D]\expec[\xi(h)]-\expec\Big[\Big(B+\sum_{i=1}^D \xi(h_i)\Big)^2\Big] \nn\\
	&\geq B + \expec[D]\expec[\xi(h)]-B e_6-\expec[D(D-1)]\expec[\xi(h)]^2-\expec[D]C_2\expec[\xi(h)]^{2\wedge(\tau-2)}\nn\\
	&=B + (\expec[D]-e_7)\expec[\xi(h)]-B e_6,
	\end{align}
with $a\wedge b$ denoting the minimum of $a$ and $b$, because $\expec[\xi(h)]$ converges to
zero for both limits of interest.

We again first perform the analysis for $\boldsymbol{\beta}$ and $\tau\neq 5$.
We get from \eqref{eq-LowerExi} in Proposition~\ref{prop-LowerExi} that
	\eq
	\expec[\xi(h_0)]\geq\left(\frac{\betahat\nu-1}{c_1}\right)^{1/(\boldsymbol{\delta}-1)}
	\geq \left(\frac{\nu(1-\betahat^2)}	{c_1}\right)^{\boldsymbol{\beta}}(\beta-\beta_c)^{\boldsymbol{\beta}},
	\en
where the last inequality holds because, by Taylor's theorem,
	\eq
	\label{eq-Taylorbetahatlow}
	\betahat \nu-1 \geq \nu(1-\betahat^2)(\beta-\beta_c).
	\en
Hence,
	\eq
	\liminf_{\beta\searrow \beta_c}\frac{M(\beta,0^+)}{(\beta-\beta_c)^{\boldsymbol{\beta}}} \geq
	\expec[D]\left(\frac{\nu(1-\betahat^2)}	{c_1}\right)^{\boldsymbol{\beta}}>0.
	\en
For $\tau=5$, we note that \eqref{eq-LowerExi-tau5} as well as the fact that $\log{1/x}\leq A_{\vep} x^{-\vep}$
for all $x\in (0,1)$ and some $A_{\vep}>0$, yields that
	\eq
	\label{log-corr-tau5-beta1}
	\expec[\xi(h_0)]\geq\left(\frac{\betahat\nu-1}{A_{\vep} c_1}\right)^{1/(2\ch{+}\vep)}
	\geq \left(\frac{\nu(1-\betahat^2)}	{A_{\vep} c_1}\right)^{1/(2\ch{+}\vep)}(\beta-\beta_c)^{1/(2\ch{+}\vep)}.
	\en
Then again using  \eqref{eq-LowerExi-tau5} yields, for some constant $c>0$,
	\eq
	\label{log-corr-tau5-beta2}
	\expec[\xi(h_0)]\geq\left(\frac{\betahat\nu-1}{c_1\log(1/\expec[\xi(h_0)])}\right)^{1/2}
	\geq c\Big(\frac{\beta-\beta_c}{\log(1/(\beta-\beta_c))}\Big)^{1/2},
	\en
once more since $x\mapsto 1/(\log(1/x))$ is increasing.

We continue with the analysis for $\boldsymbol{\delta}$.
Again, setting $\beta=\beta_c$ in~\eqref{eq-LowerExi}, we get
	\eq
	\expec[\xi(h_c)]\geq\left(\frac{\betahat_c-e_1}{c_1}\right)^{1/\boldsymbol{\delta}}B^{1/\boldsymbol{\delta}},
	\en
from which it follows that
	\eq
	\liminf_{B\searrow0} \frac{M(\beta_c,B)}{B^{1/\boldsymbol{\delta}}} \geq
	\expec[D]\left(\frac{\betahat_c}{c_1}\right)^{1/\boldsymbol{\delta}}>0,
	\en
as required.  The extension to $\tau=5$ can be dealt with in an identical way as in
\eqref{log-corr-tau5-beta1}--\eqref{log-corr-tau5-beta2}.
This proves the theorem.
\end{proof}

\section{Critical exponents: Susceptibility}
\label{sec-CritExpChi}
In this section, we study the susceptibility. In Section \ref{sec-gamma} we identify
$\boldsymbol{\gamma}$, in Section  \ref{sec-gamma'} we prove a lower bound on
$\boldsymbol{\gamma'}$ and add a heuristic why this is the correct value.

\subsection{The critical exponent $\boldsymbol{\gamma}$}
\label{sec-gamma}

For the susceptibility in the {\em subcritical} phase, i.e., in the high-temperature region $\beta<\beta_c$,
we can not only identify the critical exponent $\boldsymbol{\gamma}$, but we can also identify the constant:

\begin{theorem}[Critical exponent $\boldsymbol{\gamma}$]
For $\ch{\expec[K]<\infty}$ and $\beta<\beta_c$,
	\ch{
	\eq
	\label{chi-comp-highT}
	\chi(\beta,0^+)=1+\frac{\expec[D]\betahat}{1-\nu\betahat}.
	\en
In particular,
	\eq
	\label{chi-asy-highT}
	\lim_{\beta \nearrow \beta_c} \chi(\beta,0^+)(\beta_c-\beta) = \frac{\expec[D]\betahat_c}{1-\betahat_c^2},
	\en}
and hence
	\eq
	\boldsymbol{\gamma}=1.
	\en
\end{theorem}

\begin{proof} The proof \ch{is divided into three steps. We first reduce the suceptibility on the random graph to the one on the random Bethe
tree. Secondly, we rewrite the susceptibility on the tree using transfer matrix techniques. Finally, we use this rewrite (which applies to
\emph{all} $\beta$ and $B>0$) to prove that $\boldsymbol{\gamma}=1$.}

\paragraph{Reduction to the random tree.}
Let $\phi$ denote a vertex selected uniformly at random from $[n]$ and let $\expec_\phi$ denote its expectation. Then we can write the susceptibility as
	\eq
	\chi_n \equiv \frac1n \sum_{i,j=1}^n \Big(\langle\sigma_i\sigma_j\rangle_{\col{\mu_n}} - \langle\sigma_i\rangle_{\col{\mu_n}}\langle\sigma_j\rangle_{\col{\mu_n}}\Big)
	= \expec_\phi\left[\sum_{j=1}^n
	\Big(\langle\sigma_\phi\sigma_j\rangle_{\col{\mu_n}} - \langle\sigma_\phi\rangle_{\col{\mu_n}}\langle\sigma_j\rangle_{\col{\mu_n}}\Big)\right].
	\en
Note that
	\eq
	\label{eq-correlisderiv}
	\langle\sigma_i\sigma_j\rangle_{\col{\mu_n}} - \langle\sigma_i\rangle_{\col{\mu_n}}\langle\sigma_j\rangle_{\col{\mu_n}}
	= \frac{\partial \langle\sigma_i\rangle_{\col{\mu_n}}}{\partial B_j},
	\en
which is, by the GHS inequality~\cite{GriHurShe70}, decreasing in external fields at all other vertices $k\in[n]$. Denote by $\langle\cdot\rangle^{t,+/f}$ the Ising model with $+$/free boundary conditions, respectively, at all vertices at graph distance $t$ from $\phi$. Then, for all $t\geq1$,
	\eq
	\chi_n \geq \expec_\phi\left[\sum_{j=1}^n \Big(\langle\sigma_\phi\sigma_j\rangle^{t,+}_{\chs{\mu_n}}
	- \langle\sigma_\phi\rangle^{t,+}_{\chs{\mu_n}}\langle\sigma_j\rangle_{\chs{\mu_n}}^{t,+}\Big)\right].
	\en
By introducing boundary conditions, only vertices in the ball $B_\phi(t)$ contribute to the sum.
Hence, by taking the limit $n\rightarrow \infty$ and using that the graph is locally tree-like,
	\eq
	\chi \geq \expec\left[\sum_{j\in T_t} \Big(\langle\sigma_\phi\sigma_j\rangle^{t,+}
	- \langle\sigma_\phi\rangle^{t,+}\langle\sigma_j\rangle^{t,+}\Big)\right],
	\en
where the expectation now is over the random tree $T_t \sim \calT(D,K,t)$ with root $\phi$.

For an upper bound on $\chi_n$ we use a trick similar to one used in the proof of~\cite[Corollary~4.5]{DemMon10}: Let $B_j'=B$ if $j\in B_t(\phi)$ and $B_j'=B+H$ if $j \notin B_t(\phi)$ for some $H>-B$. Denote by $\langle\cdot\rangle_{H}$ the associated Ising expectation. Then, because of~\eqref{eq-correlisderiv},
	\eq
	\expec_\phi\left[\sum_{j\notin B_t(\phi)} \Big(\langle\sigma_\phi\sigma_j\rangle
	- \langle\sigma_\phi\rangle\langle\sigma_j\rangle\Big)\right]
	= \expec_\phi\left[ \frac{\partial}{\partial H} \langle\sigma_\phi\rangle_H \Bigg|_{H=0}\right],
	\en
By the GHS inequality, $\langle\sigma_\phi\rangle_H$ is a concave function of $H$ and hence,
	\eq
	\expec_\phi\left[\frac{\partial}{\partial H} \langle\sigma_\phi\rangle_H \Bigg|_{H=0}\right]
	\leq \expec_\phi\left[\frac{2}{B}\left(\langle\sigma_\phi\rangle_{H=0}-\langle\sigma_\phi\rangle_{H=-B/2}\right)\right].
	\en
Using the GKS inequality this can be bounded from above by
	\eq
	\expec_\phi\left[\frac{2}{B}\left(\langle\sigma_\phi\rangle^{t,+}_{H=0}-\langle\sigma_\phi\rangle^{t,f}_{H=-B/2}\right)\right]
	= \expec_\phi\left[\frac{2}{B}\left(\langle\sigma_\phi\rangle^{t,+}-\langle\sigma_\phi\rangle^{t,f}\right)\right],	
	\en
where the equality holds because the terms depend only on the system in the ball $B_t(\phi)$ and hence not on $H$. By letting $n\rightarrow\infty$, by the locally tree-likeness, this is equal to
	\eq
	\frac{2}{B}\expec\left[\left(\langle\sigma_\phi\rangle^{t,+}-\langle\sigma_\phi\rangle^{t,f}\right)\right],
	\en
where the expectation and the Ising model now is over the random tree $T_t \sim \calT(D,K,t)$ with root $\phi$. From~\cite[Lemma~3.1]{DomGiaHof10} we know that this expectation can be bounded from above by $M/t$ for some constant $M=M(\beta,B)<\infty$. Hence, if $t\rightarrow\infty$,
	\eq
	\lim_{t\rightarrow\infty}\expec\left[\sum_{j\in T_t} \Big(\langle\sigma_\phi\sigma_j\rangle^{t,+}
	- \langle\sigma_\phi\rangle^{t,+}\langle\sigma_j\rangle^{t,+}\Big)\right]
	\leq \chi \leq \lim_{t\rightarrow\infty} \expec\left[\sum_{j \in T_t} \Big(\langle\sigma_\phi\sigma_j\rangle^{t,f}
	- \langle\sigma_\phi\rangle^{t,f}\langle\sigma_j\rangle^{t,f}\Big)\right].
	\en

\paragraph{Rewrite of the susceptibility on trees.}
It remains to study the susceptibility on trees. For this, condition on the tree $T_\infty$. Then, for some vertex $j$ at height $\ell\leq t$ in the tree, denote the vertices on the unique path from $\phi$ to $j$ by $\phi=v_0,v_1,\ldots,v_\ell=j$ and let, for $0\leq i\leq\ell$, $S_{\leq i}=(\sigma_{v_0},\ldots,\sigma_{v_i})$. We first compute the expected value of a spin $\sigma_{v_i}$ on this path, conditioned on the spin values $S_{\leq i-1}$. Note that under this conditioning the expected spin value only depends on the spin value $\sigma_{v_{i-1}}$ and the effective field $h_{v_i}=h_{v_i}^{t,+/f}$ obtained by pruning the tree at vertex $v_i$, i.e., by removing all edges at vertex $v_i$ going away from the root and replacing the external magnetic field at vertex $v_i$ by $h_{v_i}$ which can be exactly computed using~\cite[Lemma~4.1]{DemMon10}. Hence,
	\eq
	\langle\sigma_{v_i} | S_{\leq i-1}\rangle^{t,+/f} =\frac{\e^{\beta\sigma_{v_{i-1}}+h_{v_i}}
	-\e^{-\beta \sigma_{v_{i-1}} - h_{v_i}}}{\e^{\beta\sigma_{v_{i-1}}+h_{v_i}}+\e^{-\beta \sigma_{v_{i-1}} - h_{v_i}}}.
	\en
We can write the indicators $\ind_{\{\sigma_{v_{i-1}}=\pm1\}}=\frac12 (1\pm\sigma_{v_{i-1}})$, so that the above equals
	\begin{align}
	\frac12& (1+\sigma_{v_{i-1}})\frac{\e^{\beta+h_{v_i}}-\e^{-\beta - h_{v_i}}}{\e^{\beta+h_{v_i}}+\e^{-\beta- h_{v_i}}}
	+\frac12 (1-\sigma_{v_{i-1}})\frac{\e^{-\beta+h_{v_i}}-\e^{\beta - h_{v_i}}}{\e^{-\beta+h_{v_i}}+\e^{\beta  - h_{v_i}}}\\
	&= \sigma_{v_{i-1}} \frac12 \left(\frac{\e^{\beta+h_{v_i}}-\e^{-\beta - h_{v_i}}}{\e^{\beta+h_{v_i}}
	+\e^{-\beta- h_{v_i}}}-\frac{\e^{-\beta+h_{v_i}}-\e^{\beta - h_{v_i}}}{\e^{-\beta+h_{v_i}}+\e^{\beta  - h_{v_i}}}\right)
	+ \frac12 \left(\frac{\e^{\beta+h_{v_i}}-\e^{-\beta - h_{v_i}}}{\e^{\beta+h_{v_i}}
	+\e^{-\beta- h_{v_i}}}+\frac{\e^{-\beta+h_{v_i}}-\e^{\beta - h_{v_i}}}{\e^{-\beta+h_{v_i}}+\e^{\beta  - h_{v_i}}}\right).\nn
	\end{align}
By pairwise combining the terms over a common denominator the above equals
	\begin{align}
	\sigma_{v_{i-1}} \frac12 & \frac{(\e^{\beta+h_{v_i}}-\e^{-\beta - h_{v_i}})
	(\e^{-\beta+h_{v_i}}+\e^{\beta  - h_{v_i}})-(\e^{-\beta+h_{v_i}}-\e^{\beta - h_{v_i}})
	(\e^{\beta+h_{v_i}}+\e^{-\beta- h_{v_i}})}{(\e^{\beta+h_{v_i}}+\e^{-\beta- h_{v_i}})(\e^{-\beta+h_{v_i}}+\e^{\beta  - h_{v_i}})} \nn\\
	&+ \frac12 \frac{(\e^{\beta+h_{v_i}}-\e^{-\beta - h_{v_i}})(\e^{-\beta+h_{v_i}}+\e^{\beta  - h_{v_i}})
	+(\e^{-\beta+h_{v_i}}-\e^{\beta - h_{v_i}})(\e^{\beta+h_{v_i}}+\e^{-\beta- h_{v_i}})}{(\e^{\beta+h_{v_i}}+\e^{-\beta- h_{v_i}})
	(\e^{-\beta+h_{v_i}}+\e^{\beta  - h_{v_i}})}.
	\end{align}
By expanding all products, this equals, after cancellations,
	\begin{align}\
	\label{eq-sigmavigivenslei}
	\sigma_{v_{i-1}}&\frac{\e^{2\beta}+\e^{-2\beta}}{\e^{2\beta}+\e^{-2\beta}+\e^{2h_{v_i}}+\e^{-2h_{v_i}}}
	+\frac{\e^{2h_{v_i}}+\e^{-2h_{v_i}}}{\e^{2\beta}+\e^{-2\beta}+\e^{2h_{v_i}}+\e^{-2h_{v_i}}}\nn\\
	&=\sigma_{v_{i-1}}\frac{\sinh(2\beta)}{\cosh(2\beta)+\cosh(2h_{v_i})} + \frac{\sinh(2h_{v_i})}{\cosh(2\beta)+\cosh(2h_{v_i})}.
	\end{align}
Using this, we have that
	\eq
	\langle\sigma_{v_\ell}\rangle^{t,+/f} =\langle\langle\sigma_{v_\ell}|S_{\leq \ell-1}\rangle^{t,+/f}\rangle^{t,+/f}
	= \langle\sigma_{v_{\ell-1}}\rangle^{t,+/f}\frac{\sinh(2\beta)}{\cosh(2\beta)
	+\cosh(2h_{v_\ell})} + \frac{\sinh(2h_{v_\ell})}{\cosh(2\beta)+\cosh(2h_{v_\ell})}.
	\en
Applying this recursively, we get
	\begin{align}
	\langle\sigma_{v_\ell}\rangle^{t,+/f}
	= \langle\sigma_{v_0}\rangle^{t,+/f} & \prod_{i=1}^{\ell}\frac{\sinh(2\beta)}{\cosh(2\beta)+\cosh(2h_{v_i})} \nn\\
	&+ \sum_{i=1}^{\ell}\left(\frac{\sinh(2h_{v_i})}{\cosh(2\beta)+\cosh(2h_{v_i})}\prod_{k=i+1}^{\ell}
	\frac{\sinh(2\beta)}{\cosh(2\beta)+\cosh(2h_{v_k})}\right).
	\end{align}
Similarly,
	\begin{align}
	\langle\sigma_{v_0} \sigma_{v_\ell}\rangle^{t,+/f}
	&= \left\langle\sigma_{v_0}\left(\sigma_{v_0}\prod_{i=1}^{\ell}\frac{\sinh(2\beta)}{\cosh(2\beta)+\cosh(2h_{v_i})} \right.\right. \nn\\
	&\left.\left.\qquad + \sum_{i=1}^{\ell}\left(\frac{\sinh(2h_{v_i})}{\cosh(2\beta)+\cosh(2h_{v_i})}\prod_{k=i+1}^{\ell}
	\frac{\sinh(2\beta)}{\cosh(2\beta)+\cosh(2h_{v_k})}\right)\right)\right\rangle^{t,+/f}\nn\\
	&=\prod_{i=1}^{\ell}\frac{\sinh(2\beta)}{\cosh(2\beta)+\cosh(2h_{v_i})}\nn\\
	&\qquad +\langle\sigma_{v_0}\rangle^{t,+/f}\sum_{i=1}^{\ell}\left(\frac{\sinh(2h_{v_i})}
	{\cosh(2\beta)+\cosh(2h_{v_i})}\prod_{k=i+1}^{\ell}\frac{\sinh(2\beta)}{\cosh(2\beta)+\cosh(2h_{v_k})}\right).
	\end{align}
Combining the above yields
	\eq
	\label{eq-exactcorrintree}
	\langle\sigma_{v_0}\sigma_{v_\ell}\rangle^{t,+/f}-\langle\sigma_{v_0}\rangle^{t,+/f}	
	\langle\sigma_{v_\ell}\rangle^{t,+/f}
	=\left(1-\left(\langle\sigma_{v_0}\rangle^{t,+/f}\right)^2\right)\prod_{i=1}^{\ell}\frac{\sinh(2\beta)}{\cosh(2\beta)+\cosh(2h_{v_i})}.
	\en
By taking the limit $t\rightarrow\infty$, we obtain
	\eq
	\chi =\expec\left[\sum_{j \in T_{\infty}} \left(1-\langle\sigma_{v_0}\rangle^2\right)
	\prod_{i=1}^{|j|}\frac{\sinh(2\beta)}{\cosh(2\beta)+\cosh(2h_{v_i})}\right].
	\en
Finally, we can rewrite
	 \eq
	\frac{\sinh(2\beta)}{\cosh(2\beta)+\cosh(2h_{v_i})}
	= \frac{2\sinh(\beta)\cosh(\beta)}{2\cosh(\beta)^2-1+\cosh(2h_{v_i})}=
	\frac{\betahat}{1+\frac{\cosh(2h_{v_i})-1}{2\cosh(\beta)^2}},
	\en
so that
	\eq
	\label{chi-rewrite}
	\chi(\beta,B)=\expec\left[\left(1-\langle\sigma_{v_0}\rangle^2\right) \sum_{j \in T_{\infty}} \betahat^{|j|}
	\prod_{i=1}^{|j|}\Big(1+\frac{\cosh(2h_{v_i})-1}{2\cosh(\beta)^2}\Big)^{-1}\right].
	\en
The rewrite in \eqref{chi-rewrite} is valid for all $\beta$ and $B>0$, and provides the starting point for
all our results on the susceptibility.

\paragraph{Identification of the susceptibility \col{for $\beta<\beta_c$}.}
We take the limit $B\searrow0$, for $\beta<\beta_c$, and apply dominated convergence.
First of all, all fields $h_i$ converge to zero by the definition of $\beta_c$, so we have pointwise convergence.
Secondly,  $1+\frac{\cosh(2h_{v_i})-1}{2\cosh(\beta)^2}\geq 1$, so that the random variable in the
expectation is bounded from above by $\sum_{j \in T_{\infty}} \betahat^{|j|}$, which has finite
expectation as we show below.
Thus, by dominated convergence, the above converges to
	\eq
	\lim_{B\searrow0}\chi(\beta,B)
	=\expec\left[\sum_{j \in T_{\infty}} \betahat^{|j|}\right].
	\en
Denote by $Z_\ell$ the number of vertices at distance $\ell$ from the root. Then,
	\eq
	\expec\left[\sum_{j\in T_\infty} \betahat^{|j|}\right]
	= \expec\left[\sum_{\ell=0}^\infty Z_\ell \betahat^{\ell}\right]=\sum_{\ell=0}^\infty \expec[Z_\ell] \betahat^{\ell},
	\en
because $Z_\ell\geq0$, a.s. Note that $Z_\ell / (\expec[D] \nu^{\ell-1})$ is a martingale, because the offspring of the root has expectation $\expec[D]$ and all other vertices have expected offspring $\nu$. Hence,
	\eq
	\ch{\lim_{B\searrow0}\chi(\beta,B)=}\sum_{\ell=0}^\infty \expec[Z_\ell] \betahat^{\ell}
	= 1+\sum_{\ell=1}^\infty \ch{\expec[D]}\nu^{\ell-1}\betahat^\ell
	= 1+\frac{\expec[D]\betahat}{1-\betahat\nu}.
	\en
\ch{This proves \eqref{chi-comp-highT}. We continue to prove \eqref{chi-asy-highT},}
which follows by using~\eqref{eq-Taylorbetahatup} and~\eqref{eq-Taylorbetahatlow}:
	\eq
	\frac{\expec[D]\betahat}{1-\betahat^2}(\beta_c-\beta)^{-1} +1
	\leq \frac{\expec[D]}{\nu}\frac{1}{1-\betahat\nu}
	\leq \frac{\expec[D]\betahat}{1-\betahat_c^2}(\beta_c-\beta)^{-1}+1.
	\en
\end{proof}

\subsection{Partial results for the critical exponent $\boldsymbol{\gamma'}$}
\label{sec-gamma'}

For the supercritical susceptibility, we prove the following lower bound on $\boldsymbol{\gamma'}$:

\begin{proposition}[Critical exponent $\boldsymbol{\gamma'}$]
\ch{For $\tau\in (3,5]$ or $\expec[K^3]<\infty$,}
	\eq
	\boldsymbol{\gamma'}\geq 1.
	\en
\end{proposition}

\begin{proof} We start by rewriting the susceptibility in a form that is convenient in the low-temperature phase.

\paragraph{A rewrite of the susceptibility in terms of i.i.d.\ random variables.}
For $\beta>\beta_c$ we start from~\eqref{chi-rewrite}.
We further rewrite
	\eq
	\chi(\beta,B)=
	\sum_{\ell=0}^\infty \betahat^{\ell} \expec\left[(1-\langle\sigma_{v_0}\rangle^2) \sum_{v_{\ell} \in T_{\infty}}
	\exp\Big\{-\sum_{i=1}^{\ell}\log\Big(1+\frac{\cosh(2h_{v_i})-1}{2\cosh(\beta)^2}\Big)\Big\}\right].
	\en
\ch{Here, and in the sequel, we use the convention that empty products, arising when $\ell=0$,
equal 1, while empty sums equal 0. Thus, the contribution due to $\ell=0$ in the above sum equals 1.}
We write $v_0=\phi$ and $v_i=a_0\cdots a_{i-1}\in {\mathbb N}^i$ for $i\geq 1$, so that $v_i$ the
$a_{i-1}$st child of $v_{i-1}$. Then,
	\eq
	\chi(\beta,B)=
	\sum_{\ell=0}^\infty \betahat^{\ell} \sum_{a_0, \ldots, a_{\ell-1}}
	\expec\left[(1-\langle\sigma_{v_0}\rangle^2)\indic{v_{\ell} \in T_{\infty}}
	\exp\Big\{-\sum_{i=1}^{\ell}\log\Big(1+\frac{\cosh(2h_{v_i})-1}{2\cosh(\beta)^2}\Big)\Big\}\right].
	\en
Let $K_{v_i}$ be the number of children of $v_i$, and condition on $K_{v_i}=k_i$ for every $i\in[0,\ell-1]$,
where we abuse notation to write $[0,m]=\{0,\ldots, m\}$.
As a result, we obtain that
	\begin{align}
	\chi(\beta,&B)=
	\sum_{\ell=0}^\infty \betahat^{\ell} \sum_{a_0, \ldots, a_{\ell-1}}\sum_{k_0,\ldots, k_{\ell-1}}
	\prob(v_{\ell}\in T_{\infty}, K_{v_i}=k_i ~\forall i\in[0, \ell-1])\\
	& \times
	\expec\left[(1-\langle\sigma_{v_0}\rangle^2)
	\exp\Big\{-\sum_{i=1}^{\ell}\log\Big(1+\frac{\cosh(2h_{v_i})-1}{2\cosh(\beta)^2}\Big)\Big\}
	\mid v_{\ell}\in T_{\infty}, K_{v_i}=k_i ~\forall i\in[0, \ell-1]\right].\nn
	\end{align}
Note that
	\eq
	\prob(K_{v_i}=k_i ~\forall i\in[0, \ell-1], v_{\ell}\in T_{\infty})
	=\prob(D=k_0)\indic{a_0\leq k_0}\prod_{i=1}^{\ell-1} \prob(K=k_i)\indic{a_i\leq k_i}.
	\en
Let $T_{i,j}$ be the tree that descibes all descendants of the $j$th child of $v_i$, with the $a_i$th child removed,
and $T_{\ell}$ the offspring of $v_{\ell}$.
When $v_{\ell}\in T_{\infty}$, all information of the tree $T_{\infty}$ can be encoded in the
collection of trees $(T_{i,j})_{j\in [0,K_{v_i}-1],i\in [0,\ell-1]}$ and $T_{\ell}$,
together with the sequence $(a_i)_{i=0}^{\ell-1}$. Denote $\vec{T}=\big((T_{i,j})_{j\in [0,K_{v_i}-1],i\in [0,\ell-1]}, T_{\ell}\big)$.
Then, for any collection of trees $\vec{t}=\big((t_{i,j})_{j\in [0,k_i-1],i\in [0,\ell-1]}, t_{\ell}\big)$,
	\eq
	\prob(\vec{T}=\vec{t}\mid K_{v_i}=k_i ~\forall i\in[0, \ell-1], v_{\ell}\in T_{\infty})
	=\prob(T=t_{\ell}) \prod_{(i,j)\in [0,k_i-1]\times [0,\ell-1]} \prob(T=t_{i,j}),
	\en
where the law of $T$ is that of a Galton-Watson tree with offspring distribution $K$.
We conclude that
	\begin{align}
	\chi(\beta,B)&=
	\sum_{\ell=0}^\infty \betahat^{\ell} \sum_{a_0, \ldots, a_{\ell-1}}\sum_{k_0,\ldots, k_{\ell-1}}
	\prob(D=k_0)\indic{a_0\leq k_0}\prod_{i=1}^{\ell-1} \prob(K=k_i)\indic{a_i\leq k_i}\\
	&\qquad \times
	\expec\left[(1-\langle\sigma_{v_0}\rangle^2)
	\exp\Big\{-\sum_{i=1}^{\ell}\log\Big(1+\frac{\cosh(2h_{i}^{\star}(\vec{k}))-1}{2\cosh(\beta)^2}\Big)\Big\}
	\right],\nn
	\end{align}
where $(h_i^{\star}(\vec{k}))_{i=0}^{\ell}$ satisfy the recursion relations $h_{\ell}^{\star}=h_{\ell,1}$
	\eq
	h_i^{\star}(\vec{k})=B+\xi(h_{i+1}^{\star}(\vec{k}))+\sum_{j=1}^{k_i-1} \xi(h_{i,j}),
	\en
and where $(h_{i,j})_{i\in[0,\ell], j\geq 1}$ are i.i.d.\ copies of
the random variable $h(\beta,B)$. We note that the law of $(h_i^{\star}(\vec{k}))_{i=0}^{\ell}$ does not depend
on $(a_i)_{i\in [0,\ell-1]}$, so that the summation over $(a_i)_{i\in [0,\ell-1]}$ yields
	\begin{align}
	\chi(\beta,B)&=
	\sum_{\ell=0}^\infty \betahat^{\ell}\sum_{k_0,\ldots, k_{\ell-1}}
	k_0\prob(D=k_0)\prod_{i=1}^{\ell-1} k_i\prob(K=k_i)\\
	&\qquad \times
	\expec\left[(1-\langle\sigma_{v_0}\rangle^2)
	\exp\Big\{-\sum_{i=1}^{\ell}\log\Big(1+\frac{\cosh(2h_{i}^{\star}(\vec{k}))-1}{2\cosh(\beta)^2}\Big)\Big\}
	\right].\nn
	\end{align}
For a random variable $X$ on the non-negative integers with $\expec[X]>0$,
we let $X^{\star}$ be the size-biased distribution of $X$ given by
	\eq
	\prob(X^{\star}=k)=\frac{k}{\expec[X]}\prob(X=k).
	\en
Then
	\begin{align}
	\chi(\beta,B)&=
	\frac{\expec[D]}{\nu}
	\sum_{\ell=0}^\infty (\betahat\nu)^{\ell}\sum_{k_0,\ldots, k_{\ell-1}}
	\prob(D^{\star}=k_0)\prod_{i=1}^{\ell-1} \prob(K^{\star}=k_i)\\
	&\qquad \times
	\expec\left[(1-\langle\sigma_{v_0}\rangle^2)
	\exp\Big\{-\sum_{i=1}^{\ell}\log\Big(1+\frac{\cosh(2h_{i}^{\star}(\vec{k}))-1}{2\cosh(\beta)^2}\Big)\Big\}
	\right].\nn
	\end{align}
Define $(h_i^{\star})_{i=0}^{\ell}=\big(h_i^{\star}(D^{\star}, K^{\star}_1, \ldots, K^{\star}_{\ell-1}, K_{\ell})\big)_{i=0}^{\ell}$,
where the random variables $(D^{\star}, K^{\star}_1, \ldots, K^{\star}_{\ell-1}, K_{\ell})$ are independent.
Then we finally arrive at
	\begin{align}
	\chi(\beta,B)&=
	\frac{\expec[D]}{\nu}
	\sum_{\ell=0}^\infty (\betahat\nu)^{\ell}
	\expec\left[(1-\langle\sigma_{v_0}\rangle^2)
	\exp\Big\{-\sum_{i=1}^{\ell}\log\Big(1+\frac{\cosh(2h_{i}^{\star})-1}{2\cosh(\beta)^2}\Big)\Big\}
	\right].
	\end{align}

\paragraph{Reduction to second moments.}
We now proceed towards the lower bound on $\boldsymbol{\gamma'}$.
\ch{Note that, a.s.,
	\eq
	\langle\sigma_{v_0}\rangle=\tanh(h_{v_0}^{\star}),
	\en
where
	\eq
	h_{v_0^{\star}}=B+\xi(h_{v_1}^{\star})+\sum_{j=1}^{D^{\star}-1} \xi(h_{0,j})\leq B+\beta+\sum_{j=1}^{D^{\star}-1} \xi(h_{0,j}).
	\en
Therefore, 	
	\eq
	\langle\sigma_{v_0}\rangle\leq \tanh(B+\beta+\sum_{j=1}^{D^{\star}-1} \xi(h_{0,j})).
	\en
The right hand side is independent of $(h_i^{\star})_{i=1}^{\ell}$, so that the expectation factorizes.
Further,
	\eq
	\expec\Big[\tanh(B+\beta+\sum_{j=1}^{D^{\star}-1} \xi(h_{0,j}))\Big]\to \tanh(\beta)=\betahat<1,
	\en
as $B\searrow 0,\beta\searrow\beta_c$.
}
Further, we restrict the sum over all $\ell$ to $\ell\leq m$, where we take $m=(\beta-\beta_c)^{-1}.$
This leads to
	\begin{align}
	\chi(\beta,B)&\geq
	\frac{(1-\betahat^2)\expec[D]}{\nu}
	\sum_{\ell=0}^m (\betahat\nu)^{\ell}
	\expec\left[
	\exp\Big\{-\sum_{i=1}^{\ell}\log\Big(1+\frac{\cosh(2h_{i}^{\star})-1}{2\cosh(\beta)^2}\Big)\Big\}
	\right].
	\end{align}	
We \ch{condition on} all coordinates of $(D^{\star}, K^{\star}_1, \ldots, K^{\star}_{\ell-1}, K_{\ell})$
being at most $b=(\beta-\beta_c)^{\ch{-}1/(\tau-3)}$, which has probability
	\begin{align}
	\prob(D^{\star}\leq b, K^{\star}_1\leq b, \ldots, K^{\star}_{\ell-1}\leq b, K_{\ell}\leq b)
	&\geq (1-o(1)) \prob(K^{\star}\leq b)^{m}\\
	&\geq (1-o(1)) \big(1-C_{K^{\star}}b^{-(\tau-3)}\big)^m,\nn
	\end{align}
which is uniformly bounded from below by a constant for the choices $m=(\beta-\beta_c)^{-1}$
and $b=(\beta-\beta_c)^{\ch{-}1/(\tau-3)}$. Also, we use that $\betahat\nu\geq 1$, since $\beta>\beta_c$.
This leads us to
	\eq
	\chi(\beta,B)
	\geq
	c_{\chi}
	\sum_{\ell=0}^m
	\overline{\expec}_b\left[
	\exp\Big\{-\sum_{i=1}^{\ell}\log\Big(1+\frac{\cosh(2h_{i}^{\star})-1}{2\cosh(\beta)^2}\Big)\Big\}
	\right],
	\en
where $\overline{\expec}_b$ denotes the conditional expectation
given that $D^{\star}\leq b, K^{\star}_1\leq b, \ldots, K^{\star}_{\ell-1}\leq b, K_{\ell}\leq b$.
Using that $\expec[\e^{X}]\geq \e^{\expec[X]}$, this leads us to
	\eq
	\chi(\beta,B)
	\geq
	c_{\chi}
	\sum_{\ell=0}^m
	\exp\Big\{-\sum_{i=1}^{\ell}\overline{\expec}_b\left[\log\Big(1+\frac{\cosh(2h_{i}^{\star})-1}{2\cosh(\beta)^2}\Big)
	\right]\Big\}.
	\en
Define, for $a>0$ and $x\geq 0$, the function $q(x)=\log\Big(1+a(\cosh(x)-1)\Big)$. Differentiating leads to
	\eq
	q'(x)=\frac{a\sinh(x)}{1+a(\cosh(x)-1)},
	\en
so that $q'(x)\leq \ch{C_q x/2}$ for some constant $C_q$ and all $x\geq 0$. As a result, $q(x)\leq C_qx^2/4$, so that
	\eq
	\chi(\beta,B)
	\geq
	c_{\chi}
	\sum_{\ell=0}^m
	\exp\Big\{-C_q\sum_{i=1}^{\ell}\overline{\expec}_b\left[(h_{i}^{\star})^2\right]\Big\}.
	\en
\paragraph{Second moment analysis of $h_{i}^{\star}$.}
As a result, it suffices to investigate second moments of $h_{i}^{\star}$, which we proceed with now.
We note that
	\eq
	h_{i}^{\star}=\xi(h_{i+1}^{\star})+B+\sum_{j=1}^{K_i^{\star}-1} \xi(h_{i,j}).
	\en
Taking expectations and using that $\xi(h)\leq \betahat h$ leads to
	\eq
	\overline{\expec}_b\left[h_{i}^{\star}\right]\leq \ch{\betahat}\overline{\expec}_b\left[h_{i+1}^{\star}\right]
	+B+\expec[K^{\star}-1\mid K^{\star}\leq b]\expec[\xi(h)].
	\en
Iterating this inequality until $\ch{\ell-i}$ and using that $\overline{\expec}_b\left[h_{\ell}^{\star}\right]\leq B+\nu \expec[\xi(h)]$
\ch{(since $\overline{\expec}_b[K]\leq \expec[K]$)} leads to
	\begin{align}
	\label{asy-Ehi}
	\overline{\expec}_b\left[h_{i}^{\star}\right]
	&\leq \betahat^{\ell-i}\ch{(B+\nu\expec[\xi(h)])}+\sum_{s=0}^{\ell-i-1} \betahat^s
	\big(B+\expec[K^{\star}-1\mid K^{\star}\leq b]\expec[\xi(h)]\big)\\
	&\leq \ch{\betahat^{\ell-i}(B+\nu\expec[\xi(h)])+}\frac{B+\expec[K^{\star}-1\mid K^{\star}\leq b]\expec[\xi(h)]}{1-\betahat}.\nn
	\end{align}
Similarly,
	\begin{align}
	\overline{\expec}_b\left[(h_{i}^{\star})^2\right]
	&\leq \betahat^2\overline{\expec}_b\left[(h_{i+1}^{\star})^2\right]
	+2\betahat\overline{\expec}_b\left[h_{i+1}^{\star}\right]\big(B+\expec[K^{\star}-1\mid K^{\star}\leq b]\expec[\xi(h)]\big)\\
	&\qquad+B^2 +2B\expec[K^{\star}-1\mid K^{\star}\leq b]\expec[\xi(h)]
	+\expec[(K^{\star}-1)(K^{\star}-2)\mid K^{\star}\leq b]\expec[\xi(h)]^2\nn\\
	&\qquad+\expec[K^{\star}-1\mid K^{\star}\leq b]\expec[\xi(h)^2].\nn
	\end{align}
Taking the limit $B\searrow 0$ we thus obtain
	\begin{align}
	\overline{\expec}_b\left[(h_{i}^{\star})^2\right]
	&\leq \betahat^2\overline{\expec}_b\left[(h_{i+1}^{\star})^2\right]
	+2\betahat\overline{\expec}_b\left[h_{i+1}^{\star}\right]\expec[K^{\star}-1\mid K^{\star}\leq b]\expec[\xi(h)]\\
	&\qquad
	+\expec[(K^{\star}-1)(K^{\star}-2)\mid K^{\star}\leq b]\expec[\xi(h)]^2+
	\expec[K^{\star}-1\mid K^{\star}\leq b]\expec[\xi(h)^2].\nn
	\end{align}
\ch{We start analysing the case where $\expec[K^3]<\infty$. By Theorem~\ref{thm-CritExp}, for $\expec[K^3]<\infty$,
	\eq
	\expec[\xi(h)] \leq C_0 (\beta-\beta_c)^{1/2},
	\en
for some constant $C_0$. Substituting \eqref{asy-Ehi}, and iterating in a similar fashion as in the proof of \eqref{asy-Ehi},
we obtain that, for $\expec[K^3]<\infty$,
	\eq
	\overline{\expec}_b\left[(h_{i}^{\star})^2\right]\leq C(\beta-\beta_c).
	\en
We next extend this analysis to $\tau\in(3,5).$ Note that, for every $a>0$,
	\eq
	\expec[(K^{\star})^a\mid K^{\star}\leq b]
	=\frac{\expec[K^{a+1}\indic{K\leq b}]}{\expec[K\indic{K\leq b}]},
	\en
so that, for $\tau\in(3,5)$,
	\eq
	\expec[(K^{\star})^2\mid K^{\star}\leq b]\leq \frac{C_{3,\tau}}{\expec[K\indic{K\leq b}]} b^{5-\tau},
	\en
Further, for $\tau\in(3,5)$,
	\eq
	\expec[\xi(h)] \leq C_0 (\beta-\beta_c)^{1/(3-\tau)},
	\en
and thus
	\eq
	\expec[(K^{\star})^2\mid K^{\star}\leq b] \expec[\xi(h)]^2C\leq b^{5-\tau}\expec[\xi(h)]^2
	\leq C (\beta-\beta_c)^{-(5-\tau)/(3-\tau)+2/(3-\tau)}=C(\beta-\beta_c).
	\en}
It can readily be seen that all other contributions to $\overline{\expec}_b\left[(h_{i}^{\star})^2\right]$
are of the same or smaller order. For example, when $\expec[K^2]<\infty$ and using that $1/(\tau-3)\geq 1/2$
for all $\tau\in (3,5)$,
	\eq
	\expec[K^{\star}-1\mid K^{\star}\leq b]\expec[\xi(h)^2]\leq C\expec[\xi(h)]^2
	=O(\beta-\beta_c),
	\en
while, when $\tau\in (3,4)$,
	\eq
	\expec[K^{\star}-1\mid K^{\star}\leq b]\expec[\xi(h)^2]\leq Cb^{4-\tau} \expec[\xi(h)]^{\tau-2}
	\ch{=}C (\beta-\beta_c)^{-(4-\tau)/(3-\tau)+(\tau-2)/(3-\tau)}=C(\beta-\beta_c)^2.
	\en
We conclude that
	\eq
	\overline{\expec}_b\left[(h_{i}^{\star})^2\right]\leq C(\beta-\beta_c).
	\en
\ch{Therefore,
	\eq
	\chi(\beta,B)
	\geq
	c_{\chi}
	\sum_{\ell=0}^m
	\exp\Big\{-C\ell (\beta-\beta_c) \Big\}=O((\beta-\beta_c)^{-1}\chs{)},
	\en
as required.}
	
The proof for $\tau=5$ is similar when noting that the logarithmic corrections
\ch{present in $\expec[\xi(h)]^2$ and in $\expec[(K^{\star})^2\mid K^{\star}\leq b]$} precisely cancel.
\end{proof}

We close this section by performing a heuristic argument to determine the upper bound on
$\boldsymbol{\gamma'}$. Unfortunately, as we will discuss in more detail following the heuristics,
we are currently not able to turn this analysis into a rigorous proof.

\paragraph{The upper bound on $\boldsymbol{\gamma'}$: heuristics for $\expec[K^3]<\infty$.}
We can bound from above
	\begin{align}
	\chi(\beta,B)&\leq
	\frac{\expec[D]}{\nu}
	\sum_{\ell=0}^\infty (\betahat\nu)^{\ell}
	\expec\left[\exp\Big\{-\sum_{i=1}^{\ell}\log\Big(1+\frac{\cosh(2h_{i}^{\star})-1}{2\cosh(\beta)^2}\Big)\Big\}
	\right].
	\end{align}
Now, the problem is that $\betahat\nu>1$ when $\beta>\beta_c$, so that we need to extract
extra decay from the exponential term, \ch{which is technically demanding, and requires us to know
various constants rather precisely. Let us show this heuristically.} It suffices to study large
values of $\ell$, since small values  can be bounded in a simple way.

We blindly put the expectation in the exponential, and Taylor expand to obtain that
	\begin{align}
	\label{expon-bd}
	\chi(\beta,B)&\approx
	\frac{\expec[D]}{\nu}
	\sum_{\ell=0}^\infty (\betahat\nu)^{\ell}
	\exp\Big\{-\sum_{i=1}^{\ell}\frac{\expec\left[(h_{i}^{\star})^2\right]}{\cosh(\beta)^2}\Big\}.
	\end{align}
We compute that
	\eq
	\cosh(\beta)^2=\frac{1}{1-\betahat^2}.
	\en
Since
	\eq
	h_{i}^{\star}\approx \betahat h_{i+1}^{\star}+\sum_{j=1}^{K_i^{\star}-1} \xi(h_{i,j}),
	\en
we have
	\eq
	\expec\left[h_{i}^{\star}\right]\approx \frac{\expec[K^{\star}-1]}{1-\betahat} \expec[\xi(h)],
	\en
and
	\eq
	\expec\left[(h_{i}^{\star})^2\right]\approx \frac{2\betahat\expec[K^{\star}-1]^2
	+\expec[(K^{\star}-1)(K^{\star}-2)](1-\betahat)}{(1-\betahat^2)(1-\betahat)} \expec[\xi(h)]^2
	+\frac{\expec[K^{\star}-1]}{1-\betahat^2}\expec[\xi(h)^2].
	\en
Ignoring all error terms in the proof of Lemma \ref{lem-boundxih2}
shows that
	\eq
	\expec[\xi(h)^2]\approx \frac{\nu_2\betahat^2}{1-\betahat} \expec[\xi(h)]^2=C_2\expec[\xi(h)]^2,
	\en
so in total we arrive at (\ch{also using that $\betahat\approx 1/\nu$})
	\eq
	\label{h-star-squared}
	\expec\left[(h_{i}^{\star})^2\right]\approx
	\frac{\nu_3(1-\betahat)/\nu+3\nu_2^2/\nu^3}{(1-\betahat^2)(1-\betahat)} \expec[\xi(h)]^2.
	\en
As a result,
	\eq
	\label{h-star-squared-rep}
	\frac{\expec\left[(h_{i}^{\star})^2\right]}{\cosh(\beta)^2}\approx
	\frac{\nu_3(1-\betahat)/\nu+3\nu_2^2/\nu^3}{1-\betahat}\expec[\xi(h)]^2.
	\en
Ignoring error terms in the computation in Lemma \ref{lem-boundxih3} shows that
	\eq
	\expec[\xi(h)^3] \approx C_3 \expec[\xi(h)]^3,
	\en
where
	\eq
	C_3 = \frac{\betahat^3}{1-\betahat^3 \nu} \left(\nu_3 + 3 \nu_2 C_2\right)
	\approx \frac{\betahat^3}{1-\betahat^2} \left(\nu_3 + 3 \nu_2 C_2\right)
	=\frac{\betahat^3}{(1-\betahat^2)(1-\betahat)} \left(\nu_3(1-\betahat) + 3 \ch{(\nu_2/\nu)^2}\right),
	\en
since $\ch{\betahat\approx 1/\nu}$. Further, again ignoring error terms in
\eqref{eq-UpperExi} \ch{and Taylor expanding to third order} shows that
	\eq\label{eq-UpperExi-rep}
	\expec[\xi(h)] \approx \betahat \nu \expec[\xi(h)] - C_1 \expec[\xi(h)]^{3},
	\en
where
	\eq
	C_1=-\frac{\xi'''(0)}{6} \big(\nu C_3+3\nu_2C_2+\nu_3\big),
	\en
and $\xi'''(0)=-2\betahat(1-\betahat^2)$. Substituting the definitions for $C_2$ and $C_3$ yields
	\begin{align}	
	C_1&=\frac{\betahat(1-\betahat^2)}{3}\big(\nu C_3+3\nu_2C_2+\nu_3\big)\\
	&=\frac{\betahat}{3(1-\betahat)} \big(\nu\betahat^3\nu_3(1-\betahat)+ 3\nu\betahat^3\ch{(\nu_2/\nu)}^2+3\nu_2^2\betahat^2(1-\betahat^2)+\nu_3(1-\betahat)(1-\betahat^2)\big)\nn\\
	&=\frac{\betahat}{3(1-\betahat)} \big(\nu_3(1-\betahat)+ \ch{3\nu_2^2\betahat^2}\big).\nn
	\end{align}

Thus, we arrive at
	\eq
	\expec[\xi(h)]^2\approx \frac{\betahat \nu-1}{C_1},
	\en
so that substitution into \eqref{h-star-squared-rep} leads to
	\begin{align}
	\label{the-miracle}
	\frac{\expec\left[(h_{i}^{\star})^2\right]}{\cosh(\beta)^2}
	&\approx (\betahat \nu-1)
	\frac{\ch{3\big(\nu_3(1-\betahat)/\nu+3\nu_2^2/\nu^3\big)}}
	{\betahat\big(\nu_3(1-\betahat)+ \ch{3\nu_2^2\betahat^2}\big)}\ch{=3(\betahat \nu-1)}.
	\end{align}
%Multiplying numerator and denominator by $\nu^4=1/\betahat_c^4$ leads to
%	\begin{align}
%	\frac{\expec\left[(h_{i}^{\star})^2\right]}{\cosh(\beta)^2}
%	&\approx (\betahat \nu-1)
%	\frac{3\nu\big(\nu_3\nu(\nu-1)+3\nu_2^2\big)}
%	{\nu_3(\nu-1)\nu^2+3\nu\nu_2^2+3\nu_2^2(\nu^2-1)/\nu}\\
%	&\geq
%	%(\betahat \nu-1)
%	%\frac{3\nu\big(\nu_3\nu(\nu-1)+3\nu_2^2\big)}
%	%{\nu_3(\nu-1)\nu^2+3\nu\nu_2^2+3\nu_2^2\nu^2/\nu}\nn\\
%	%&=
%	(\betahat \nu-1)
%	\frac{3\big(\nu_3\nu(\nu-1)+3\nu_2^2\big)}
%	{\nu_3\nu(\nu-1)+6\nu_2^2},\nn
%	\end{align}
%where the constant satisfies
%	\eq
%	\label{constant-cond}
%	\frac{3\big(\nu_3\nu(\nu-1)+3\nu_2^2\big)}
%	{\nu_3\nu(\nu-1)+6\nu_2^2}>3/2.
%	\en
\ch{We conclude that
	\eq
	\label{cancel-pos-neg}
	(\betahat \nu) \exp{\big\{-\frac{\expec\left[(h_{i}^{\star})^2\right]}{\cosh(\beta)^2}\big\}}
	\leq \big(1+(\betahat \nu-1)\big)\e^{-3(\betahat \nu-1)}\leq \e^{-2(\betahat \nu-1)}.
	\en
This suggests that
	\eq
	\chi(\beta,B)\leq
	\frac{\expec[D]}{\nu}
	\sum_{\ell=0}^\infty \e^{-2\ell (\betahat \nu-1)}=O((\betahat \nu-1)^{-1}),
	\en
as required. Also, using \eqref{expon-bd}, this suggests that
	\eq
	\lim_{\beta\searrow \beta_c} (\betahat \nu-1)\chi(\beta,0^+)=\expec[D]/(2\nu),
	\en
where the constant is precisely half the one for the subcritical susceptibility (see \eqref{chi-comp-highT}).
It can be seen by an explicit computation that the same factor $1/2$ is also present in the same
form for the Curie-Weiss model.}
\col{
Indeed  for the Boltzmann\chs{-}Gibbs measure with Hamiltonian
$H_n(\sigma) = -\frac{1}{2\chs{n}}\sum_{i,j}\sigma_i\sigma_j $ one has $\beta_c=1$ and
a susceptibility $\chi(\beta,0^+)=1/(1-\beta)$ for $\beta<\beta_c$,
$\chi(\beta,0^+) = (1-m^2)/(1-\beta(1-m^2))$ with $m$ the non-zero solution
of $m=\tanh(\beta m)$ for $\beta>\beta_c$.} \chs{Expanding this gives $m^2=3(\beta-1)(1+o(1))$ for $\beta\searrow1$ and hence $\chi(\beta,0^+)=(1+o(1))/(1-\beta(1-3(\beta-1)))=(1+o(1))/(2(\beta-1))$.}
%\todo{Add a good reference to the analysis for the CW model!}

\ch{It is a non-trivial task to turn the heuristic of this Section into a proof because of several reasons:
(a) We need to be able to justify the step where we put expectations in the exponential.
While we are dealing with random variables with small means, they are not independent, so this is
demanding; (b) We need to know the constants very precisely, as we are using the fact that a
positive and negative term cancel in \eqref{cancel-pos-neg}. The analysis performed in the
previous sections does not give optimal control over these constants, so this step also requires substantial
work.}

\ch{The above heuristic does not apply to $\tau\in (3,5]$. However,  the constant in  \eqref{the-miracle}
is \emph{always} equal to 3, irrespective of the degree distribution. This suggests that also for $\tau\in(3,5]$, we should have
$\boldsymbol{\gamma'}\leq1$.}

%\paragraph{Connections to large deviations for Markov Chains.}
%We note that the sequence $X_i=h_{\ell-i}^{\star}$, for $i=0,\ldots, \ell$ forms a time-homogeneous
%Markov chain starting at $h_{\ell}^{\star}=h(\beta,B)$, with the transition rules given by
%	\eq
%	X_i=\xi(X_{i-1})+\sum_{j=1}^{S_i} Y_{i,j},
%	\en
%where $S_i=K_{\ell-i}^{\star}-1$ and $Y_{i,j}=\xi(h_{i,j})$ and the variables $S_i$ and $(Y_{i,j})_{j=1}^{S_i}$
%are independent of $X_i$. Therefore, one may assume that the sequence $(X_i)_{i\geq 0}$
%satisfies a large deviation principle. We are interested in the exponential functional
%	\eq
%	\expec\left[\exp\Big\{-\sum_{i=1}^{\ell}\log\Big(1+\frac{\cosh(2X_i)-1}{2\cosh(\beta)^2}\Big)\Big\}\right].
%	\en
%Is this of any use?

\paragraph*{Acknowledgements.}
The work of SD and RvdH is supported in part by The Netherlands Organisation for
Scientific Research (NWO).
CG acknowledge\chs{s} financial support by  the Italian Research Funding Agency (MIUR)
through the FIRB project grant n.\ RBFR10N90W.
%\todo{Add grant Cristian!}

%\listoftodos

\end{document}